\documentclass{amsart}

\usepackage[colorlinks=true, urlcolor=green, citecolor=blue, linkcolor=blue, hypertexnames=false, hyperfootnotes=true]{hyperref}
\usepackage{amssymb}
\usepackage{mathrsfs}
\usepackage{mathtools}
\usepackage{enumitem}
\usepackage{stmaryrd}

\newtheorem{cor}{Corollary}

\newtheorem{prp}{Proposition}

\newtheorem{thm}{Theorem}

\theoremstyle{definition}

\numberwithin{equation}{section}

\newcounter{ass}
\newenvironment{assumption}{\begin{enumerate}[label=(\textbf{\Alph*})]\setcounter{enumi}{\theass}}{\setcounter{ass}{\value{enumi}}
\end{enumerate}}

\author{Tristan Bice}
\address{Federal University of Bahia\\
Salvador\\
Brazil}
\email{Tristan.Bice@gmail.com}
\thanks{This research has been supported by an IMPA (Brazil) postdoctoral fellowship.}
\keywords{ring, involution, accretive, order, orthogonality, Rickart}
\subjclass[2010]{06F25, 16W10, 46L05, 47B44}

\title{*-Ring Orderings}

\begin{document}

\begin{abstract}
We examine a number of *-ring orderings, generalizing classical properties of *-positive elements to *-accretives.  We also examine *-rings satisfying versions of Blackadar's property (SP), generalizing some basic properties of Rickart *-rings to Blackadar *-rings.
\end{abstract}

\maketitle

\subsection*{Motivation}

Orderings on the positive elements have long been fundamental to operator algebra theory.  More recently, the larger class of accretive elements has become important for generalizing C*-algebra theory to Banach algebras (see \cite{BlecherRead2014}, \cite{BlecherOzawa2014}, \cite{Blecher2015} and the references therein).  Here we make some observations relating to order and orthogonality on the *-accretives in a purely algebraic context.  This demonstrates that certain basic properties only require a weak fragment of the C*-algebra structure.  It might also serve as a guide for properties to look for in more general Banach algebras.

A more traditional approach to *-rings would be to focus on just the projections and assume they correspond to annihilators (see \cite{Berberian1972}).  However, this does not apply to many (e.g. infinite dimensional separable) C*-algebras.  This motivates us to examine weaker conditions which correspond to Blackadar's property (SP) in the C*-algebra case.  This allows us to generalize some of the classical Rickart *-ring theory, as we demonstrate in \autoref{B*R} and \autoref{LS}.

\subsection*{Outline}  In \autoref{AR} we make some general definitions for binary relations.

In \autoref{S} we discuss a number of semigroup orderings.

In \autoref{*R} we review proper *-rings and define an equivalence relation from the skew-adjoints.  A weaker preorder is then defined from the *-accretives in \autoref{P}.  The only assumptions we require are \ref{unital*ring} and \ref{antisymmetry} which say that $A$ is a proper unital *-ring for which this preorder is antisymmetric on the self-adjoints.

In \autoref{B&C} we introduce some other important subsets of $A$ and discuss their interrelationships, closure properties and the order relations they define.

In \autoref{O} we generalize orthogonality properties of *-positive elements to *-accretives, e.g. showing orthogonality is symmetric on $\mathfrak{c}$ and $\mathfrak{c}\mathfrak{c}$ contains no non-zero nilpotents.

In \autoref{F} we show that the fixator relation $\ll$ is auxiliary to various other order relations and discuss lattice properties and Riesz interpolation for $\ll$.

In \autoref{PJ} we characterize projections and their products, sums and differences.

In \autoref{C} we use the extra assumption \ref{abA+} to generalize some C*-algebra results on *-positive decompositions, square-roots and products.

In \autoref{B*R} we examine the relationships between various kinds of Blackadar *-rings.

In \autoref{LS} we characterize projection supremums/infimums in $\subseteq_\perp$\!-Blackadar *-rings.

\newpage

\section{Orderings}\label{AR}

We define the usual \emph{composition} of relations $\mathbin{\ll},\mathbin{\preceq}\subseteq A\times A$ by
\begin{align*}
\ll\circ\preceq\ &=\ \bigcup_{c\in A}\{(a,b)\in A\times A:(a,c)\in\mathbin{\ll}\text{ and }(c,b)\in\mathbin{\preceq}\},\\
\text{i.e.}\quad a\ll\circ\preceq b\ &\Leftrightarrow\ \exists c\in A(a\ll c\preceq b)\quad\text{in standard infix notation}.
\end{align*}
Generalizing the definition of auxiliarity in \cite{GierzHofmannKeimelLawsonMisloveScott2003} Definition I-1.11, we say
\begin{align*}
\ll\text{ is \emph{left auxiliary} to }\preceq\quad&\Leftrightarrow\quad\mathbin{\ll}\circ\mathbin{\preceq}\ \ =\ \mathbin{\ll}\ \subseteq\ \mathbin{\preceq}.\\
\ll\text{ is \emph{right auxiliary} to }\preceq\quad&\Leftrightarrow\quad\mathbin{\preceq}\circ\mathbin{\ll}\ \ =\ \mathbin{\ll}\ \subseteq\ \mathbin{\preceq}\\
\ll\text{ is \emph{auxiliary} to }\preceq\quad&\Leftrightarrow\quad\mathbin{\ll}\text{ is left and right auxiliary to }\preceq.
\end{align*}
Define $B\ll_\mathsf{F}C\ \Leftrightarrow\ b\ll c$, for all $b\in B$ and $c\in C$, where $B,C\neq\emptyset$ are finite.
\begin{align*}
\ll\text{ is \emph{transitive}}\quad&\Leftrightarrow\quad\mathbin{\ll}\circ\mathbin{\ll}\ \subseteq\ \mathbin{\ll}.\\
\ll\text{ has \emph{interpolation}}\quad&\Leftrightarrow\quad\mathbin{\ll}\circ\mathbin{\ll}\ \supseteq\ \mathbin{\ll}.\\
\ll\text{ has \emph{Riesz interpolation}}\quad&\Leftrightarrow\quad\mathbin{\ll}_\mathsf{F}\circ\mathbin{\ll_\mathsf{F}}\ \supseteq\ \mathbin{\ll_\mathsf{F}}.
\end{align*}
So if $\ll$ is left (or right) auxiliary to $\preceq$ then $\ll$ is automatically transitive, for then $\mathbin{\ll}\circ\mathbin{\ll}\ \subseteq\ \mathbin{\ll}\circ\mathbin{\preceq}\ \subseteq\ \ll$.  Also $\ll$ is self-auxiliary iff $\ll$ is transitive and has interpolation.  Transitivity also means it suffices for Riesz interpolation to hold on pairs of elements in $A$.  Further define the following standard terminology.
\begin{align*}
\ll\text{ is \emph{reflexive}}\quad&\Leftrightarrow\quad\mathbin{=}\ \subseteq\ \mathbin{\ll}.\\
\ll\text{ is \emph{antisymmetric}}\quad&\Leftrightarrow\quad\mathbin{=}\ \supseteq\ \mathbin{\ll}\cap\mathbin{\gg}.\\
\ll\text{ is \emph{symmetric}}\quad&\Leftrightarrow\quad\mathbin{\ll}\ =\ \mathbin{\gg}.\\
\ll\text{ is a \emph{preorder}}\quad&\Leftrightarrow\quad\mathbin{\ll}\text{ is transitive and reflexive}.\\
\ll\text{ is a \emph{partial order}}\quad&\Leftrightarrow\quad\mathbin{\ll}\text{ is an antisymmetric preorder}.\\
\ll\text{ is an \emph{equivalence relation}}\quad&\Leftrightarrow\quad\mathbin{\ll}\text{ is a symmetric preorder}.
\end{align*}

Primarily for use in \autoref{B*R} and \autoref{LS}, let $a\ll$ and $\ll a$ denote the subsets defined by
\begin{align*}
a\ll\ &=\ \{b\in A:a\ll b\}.\\
\ll a\ &=\ \{b\in A:b\ll a\}.
\end{align*}
For any relation $\mathbin{\preceq}\subseteq\mathscr{P}(A)\times\mathscr{P}(A)$, where $\mathscr{P}(A)=\{B\subseteq A\}$, define
\begin{align}
\label{R^}a\preceq^\ll b\quad&\Leftrightarrow\quad(\ll a)\preceq(\ll b).\\
\label{R_}a\preceq_\ll b\quad&\Leftrightarrow\quad(b\ll)\preceq(a\ll).
\end{align}
Note transitivity, reflexivity and symmetry hold for $\preceq^\ll$ if they hold for $\preceq$.  In particular, $\subseteq^\ll$ and $\subseteq_\ll$ are preorders while $=^\ll$ and $=_\ll$ are equivalence relations.  Also, for any $\mathbin{\preceq}\subseteq A\times A$,
\begin{align}
\mathbin{\preceq}\ \subseteq\ \mathbin{\subseteq^\ll}\quad&\Leftrightarrow\quad\mathbin{\ll}\circ\mathbin{\preceq}\ \subseteq\ \mathbin{\ll}.\\
\label{preceqsubsub}\mathbin{\preceq}\ \subseteq\ \mathbin{\subseteq_\ll}\quad&\Leftrightarrow\quad\mathbin{\preceq}\circ\mathbin{\ll}\ \subseteq\ \mathbin{\ll}.
\intertext{Thus transitivity and reflexivity of $\ll$ is characterized in terms of $\subseteq^\ll$ (or $\subseteq_\ll$) by}
\nonumber\mathbin{\ll}\ \subseteq\ \mathbin{\subseteq^\ll}\quad&\Leftrightarrow\quad\mathbin{\ll}\circ\mathbin{\ll}\ \subseteq\ \mathbin{\ll}.\\
\nonumber\mathbin{\subseteq^\ll}\ \subseteq\ \mathbin{\ll}\ \ \quad&\Leftrightarrow\hspace{32pt}\mathbin{=}\ \subseteq\ \mathbin{\ll}.
\end{align}
So $\subseteq^\ll$ and $\subseteq_\ll$ are the weakest relations having $\ll$ as a left and right auxiliary respectively, as long as $\ll$ is transitive.  And they both coincide with $\ll$ if $\ll$ is a preorder.  These constructions apply to non-transitive relations too, for example to the $\ll$-incompatibility relation $\top$ defined by
\[a\mathbin{\top}b\quad\Leftrightarrow\quad\forall c\in A(c\ll a,b\ \Rightarrow\ (c\ll)=A).\]
Separativity, as in \cite{Kunen1980} Chapter 2 Exercise (15), is then naturally defined by
\[\ll\text{ is \emph{separative}}\quad\Leftrightarrow\quad\mathbin{\ll}\ =\ \mathbin{\subseteq_\top}.\]

We also define \emph{supremums} $\bigvee$ and \emph{infimums} $\bigwedge$ of $B\subseteq A$ w.r.t. $\ll$ by
\begin{align}
\label{sup}a=\bigvee B\quad&\Leftrightarrow\quad(a\ll)=\bigcap_{b\in B}(b\ll).\\
\label{inf}a=\bigwedge B\quad&\Leftrightarrow\quad(\ll a)=\bigcap_{b\in B}(\ll b).
\end{align}
So when $\ll$ is reflexive, supremums/infimums are upper/lower bounds.  If $\ll$ is also antisymmetric, then supremums/infimums are unique, when they exist.  Also $\ll$-supremums/$\ll$-infimums are precisely the $\subseteq_\ll$-supremums/$\subseteq^\ll$-infimums so we could restrict to preorders here, as is often done in the literature.  Also define
\begin{align*}
A\text{ is a \emph{$\ll$-semilattice}}\quad&\Leftrightarrow\quad\bigvee F\text{ exists for all non-empty finite }F\subseteq A.\\
A\text{ is a \emph{$\ll$-lattice}}\quad&\Leftrightarrow\quad A\text{ is a $\ll$-semilattice and $\gg$-semilattice}.
\end{align*}
For $A$ to be a $\ll$-lattice it suffices that $a\vee b=\bigvee\{a,b\}$ and $a\wedge b=\bigwedge\{a,b\}$ exist/are defined, for all $a,b\in A$.  This implies Riesz interpolation is equivalent to interpolation.  For $B\subseteq A$, we also define
\begin{align*}
B\text{ is $\ll$-\emph{cofinal} in }A\quad&\Leftrightarrow\quad B\cap(a\ll)\neq\emptyset\text{ whenever }(a\ll)\neq\emptyset.\\
B\text{ is $\ll$-\emph{coinitial} in }A\quad&\Leftrightarrow\quad B\cap(\ll a)\neq\emptyset\text{ whenever }(\ll a)\neq\emptyset.
\end{align*}

\section{Semigroups}\label{S}

In a semigroup $A$ we define the \emph{Green}, \emph{fixator} and \emph{orthogonality} relations by
\begin{align*}
a\preceq b\quad&\Leftrightarrow\quad a\in Ab.\\
a\ll b\quad&\Leftrightarrow\quad a=ab.\\
a\perp b\quad&\Leftrightarrow\quad0=ab.
\intertext{So $\perp$ requires a zero $0\in A$ i.e. satisfying $0A=\{0\}=A0$.  Also $(\preceq a)=Aa$ so}
a\subseteq^\preceq b\quad&\Leftrightarrow\quad Aa\subseteq Ab,
\end{align*}
which provides an alternative description of $\preceq$ when $\preceq$ is reflexive, e.g. when $A$ has a unit $1\in A$, i.e. satisfying $1a=a=a1$, for all $a\in A$.  In this case, the symmetrization $\mathcal{L}=\mathbin{\preceq}\cap\mathbin{\succeq}$ is well-known in semigroup theory as one of Green's relations (see \cite{Lawson2004} Chapter 10), while $\mathbin{\precapprox}=\mathbin{\preceq\circ\preceq^\mathrm{op}}$, where $a\preceq^\mathrm{op}b\Leftrightarrow a\in bA$, has been studied for C*-algebra $A$ in \cite{Cuntz1977}.  Variants of $\ll$ and $\perp$ are also often considered in C*-algebras \textendash\, see \cite{Blackadar2013} II.3.1.13 and II.3.4.3.  They also crop up naturally in lattice theory.  Indeed, if $\leq$ is a partial order making $A$ a $\geq$-semilattice and we take $\wedge$ as our semigroup operation then $\leq\ =\ \preceq\ =\ \ll$ and $\top=\ \perp$.

In general, we have the following relationships between $\preceq$, $\ll$ and $\perp$.
\begin{align}
\label{preceqpreceq}\preceq\circ\preceq\quad&\subseteq\quad\preceq.\\
\label{preceqll}\preceq\circ\ll\hspace{8pt}&\subseteq\quad\ll\quad\subseteq\quad\preceq.\\
\label{preceqperp}\preceq\circ\perp\quad&\subseteq\quad\perp.
\end{align}

\begin{proof}\
\begin{itemize}
\item[\eqref{preceqpreceq}]  If $a\preceq c\preceq b$ then $a=dc$ and $c=eb$ so $a=dc=deb$, i.e. $a\preceq b$.
\item[\eqref{preceqll}]  If $a\preceq c\ll b$ then $a=dc$ and $c=cb$ so $a=dc=dcb=ab$, i.e. $a\ll b$.
\item[\eqref{preceqperp}]  If $a\preceq c\perp b$ then $a=dc$ and $0=cb$ so $0=d0=dcb=ab$, i.e. $a\perp b$.\qedhere
\end{itemize}
\end{proof}

By \eqref{preceqll}, $\ll$ is transitive.  Applied to binary relations on $A$ under composition, this shows that auxiliarity is itself a transitive relation.  If $\preceq$ is reflexive (e.g. if $A$ is unital), \eqref{preceqll} also shows that $\preceq$ is right auxiliary to $\ll$.

We will also consider $\preceq$ relative to various subsets $B$ of $A$ defined by
\[a\preceq b\quad\Leftrightarrow\quad a\in Bb.\]
If $A$ is unital then we can characterize properties of $\preceq$ by those of $B$ as follows.
\begin{align*}
BB\subseteq B&\quad\Leftrightarrow\quad\mathbin{\preceq\circ\preceq}\ \subseteq\ \mathbin{\preceq}.\tag{Transitivity}\label{trans}\\
1\in B&\quad\Leftrightarrow\hspace{28pt}\mathbin{=}\ \subseteq\ \mathbin{\preceq}.\tag{Reflexivity}\label{reflex}\\
B^{-1}=B&\quad\Leftrightarrow\hspace{28pt}\mathbin{\succeq}\ =\ \mathbin{\preceq}.\tag{Symmetry}\label{sym}
\end{align*}
Here $B^{-1}$ denotes the inverses of invertible elements of $B$.
\begin{proof}\
\begin{itemize}
\item[\eqref{trans}]  If $\mathbin{\preceq\circ\preceq}\subseteq\mathbin{\preceq}$ and $a,b\in B$ then $ab\preceq b\preceq1$ so $ab\preceq1$, i.e. $ab\in B$.

\item[\eqref{reflex}]  If $1\in B$ then $a\in Ba$ so $a\preceq a$, for all $a\in A$.  If $1\preceq1$ then $1\in B1=B$.

\item[\eqref{sym}]  If $B^{-1}=B$ and $a\preceq b$ then $a\in Bb$ so $b\in B^{-1}a=Ba$, i.e. $b\preceq a$.

If $\mathbin{\succeq}=\mathbin{\preceq}$ and $a\in B$ then $a\preceq1$ so $1\preceq a$ and hence $1\in Ba$, i.e. $a$ has a left inverse $a^{-1}\in B$.  Likewise, $a^{-1}$ has a left inverse $(a^{-1})^{-1}\in B$.  But then $(a^{-1})^{-1}=(a^{-1})^{-1}1=(a^{-1})^{-1}a^{-1}a=1a=a$ so $a$ is a left inverse of $a^{-1}$, i.e. $a^{-1}$ is also a right inverse of $a$.\qedhere
\end{itemize}
\end{proof}

Often $\preceq$ is considered when $B$ is a subset of a group $A$.  In this case, using additive notation, any subsemigroup $B=B+B$ containing $0$ defines a preorder by
\[a\preceq b\quad\Leftrightarrow\quad b-a\in B,\]
which is an equivalence relation iff $B=-B$ is a subgroup, and a partial order iff
\[B\cap-B\subseteq\{0\}\quad\Leftrightarrow\quad\mathbin{\preceq}\cap\mathbin{\succeq}\ \subseteq\ \mathbin{=}.\tag{Antisymmetry}\]

\section{*-Rings}\label{*R}

Following \cite{Berberian1972}, we make the following standing assumption until \autoref{B*R}.
\begin{assumption}
\item\label{unital*ring} $A$ is a proper unital\footnote{Unitality is required to define the unit ball $\mathfrak{B}$ (see \autoref{B&C} below).  However, any non-unital proper *-ring has a proper unitization (see \cite{Berberian1972} \S5 Definition 3),  which could be used to generalize the theory.  The only caveat is that different unitizations might yield different generalizations.} *-ring.
\end{assumption}
So the \emph{adjoint} * is a proper self-inverse morphism from $A$ to $A^\mathrm{op}$, i.e.
\begin{align}
a^{**} &=a.\nonumber\\
(ab)^* &=b^*a^*.\nonumber\\
(a+b)^* &=a^*+b^*.\nonumber\\
\label{properness}a^*a=0\ &\Rightarrow\ a=0.
\end{align}

The \emph{self-adjoint}, \emph{skew-adjoint} and \emph{normal} elements are defined by
\begin{align*}
A_\mathrm{sa} &=\{a\in A:a=a^*\},\\
A_\mathrm{sk} &=\{a\in A:a=-a^*\}\text{ and}\\
A_\mathrm{n} &=\{a\in A:aa^*=a^*a\}\\
&\supseteq A_\mathrm{sa}\cup A_\mathrm{sk}.
\end{align*}
As $a^*+(-a)^*=(a-a)^*=0^*=(0^*0)^*=0^*0^{**}=0^*0=0$, we have $(-a)^*=-a^*$.  Thus $A_\mathrm{sk}=-A_\mathrm{sk}=A_\mathrm{sk}+A_\mathrm{sk}$ and we get an equivalence relation defined by
\[a\equiv b\quad\Leftrightarrow\quad a-b\in A_\mathrm{sk}.\]
But $a-b\in A_\mathrm{sk}$ means $a-b=-(a-b)^*=-a^*+b^*$ so 
\[a\equiv b\quad\Leftrightarrow\quad a^*+a=b^*+b.\]
So $\equiv$ is the equivalence relation coming from the $+$-homomorphism $a\mapsto a^*+a$, which induces a $+$-morphism from $A/A_\mathrm{sk}$ to $A_\mathrm{sa}$.  In particular, for all $a\in A$,
\[a\equiv a^*.\]

\section{Positivity}\label{P}
Define the \emph{*-squares}, \emph{*-sums}, \emph{*-positive} and \emph{*-accretive} elements by
\begin{align*}
|A|^2 &=\{a^*a:a\in A\}.\\
A_\Sigma &=\{\sum_{k=1}^na_k:a_1,\cdots,a_n\in|A|^2\}.\\
A_+ &=\{a\in A:na\in A_\Sigma,\text{ for some }n\in\mathbb{N}\}.\\
\mathfrak{r} &=\{a\in A:a+a^*\in A_+\}.
\end{align*}
The only other standing assumption we need until \autoref{C} is that $A_\Sigma$ is salient, i.e.
\begin{assumption}
\item\label{antisymmetry} $A_\Sigma\cap-A_\Sigma=\{0\}$.
\end{assumption}

This means that $(A,+)$ is torsion-free, for if $na=0$ then $na^*a=0$ and hence $a^*a=-(n-1)a^*a$, which means $a^*a=0$, by \ref{antisymmetry}, and hence $a=0$, by \eqref{properness}.\footnote{Conversely, \ref{antisymmetry} follows from \ref{unital*ring}, $2$ is not a zero divisor in $A$, $-1$ has a square root in $A'$ and every $a\in A_+$ has a square root in $A_\mathrm{sa}\cap\{a\}''$, where $B'=\{a\in A:ab=ba\}$ \textendash\, see \cite{Berberian1972} \S51.}
This, in turn, means that $A_+$ is salient too, for if $a\in A_+\cap-A_+$ then we have $m,n\in\mathbb{N}$ with $ma,-na\in A_\Sigma$ and hence $mna\in A_\Sigma\cap-A_\Sigma=\{0\}$ so $a=0$.  Thus
\begin{align}
\label{A+cap-A+}-A_+\cap A_+\hspace{1pt}&=\{0\}.\\
\label{AsacapAsk}A_\mathrm{sk}\cap A_\mathrm{sa}&=\{0\}.
\end{align}
It fact, for \eqref{AsacapAsk} it suffices that $2$ is not a zero-divisor, i.e. $2(0\neq)\subseteq(0\neq)$.

For $B\subseteq A$ and $n\in\mathbb{N}$, define $\tfrac{1}{n}B=\{a\in A:na\in B\}$ so
\begin{align*}
B&=\tfrac{1}{n}nB.\\
\tfrac{1}{m}\tfrac{1}{n}B&=\tfrac{1}{mn}B.\\
B\subseteq\tfrac{1}{n}B\ &\Leftrightarrow\ nB\subseteq B.
\end{align*}
By definition, $A_+=\bigcup_n\tfrac{1}{n}A_\Sigma$ so $\frac{1}{m}A_+=\bigcup_n\tfrac{1}{mn}A_\Sigma\subseteq A_+$, for all $m\in\mathbb{N}$.  If $ma,nb\in A_\Sigma$ then $mna,mnb\in A_\Sigma$ so $mn(a+b)\in A_\Sigma$, as $A_\Sigma=A_\Sigma+A_\Sigma$.  Thus
\begin{align*}
\tfrac{1}{n}A_+=A_+&=A_++A_+\quad\text{so}\\
\tfrac{1}{n}\mathfrak{r}=\mathfrak{r}&=\mathfrak{r}+\mathfrak{r}.
\end{align*}

Also $A_\mathrm{sa}=\frac{1}{n}A_\mathrm{sa}$, for if $na\in A_\mathrm{sa}$ then $na=na^*$ so $n(a-a^*)=0$ and $a=a^*$.  Certainly $|A|^2\subseteq A_\mathrm{sa}$ so $A_\Sigma\subseteq A_\mathrm{sa}$ and hence $A_+\subseteq A_\mathrm{sa}$ too, as $A_+=\bigcup_n\tfrac{1}{n}A_\Sigma$.  If $a\in A_\mathrm{sa}\cap\mathfrak{r}$ then $2a=a+a^*\in A_+$ so $a\in\frac{1}{2}A_+=A_+$.  While if $a\in\mathfrak{r}\cap-\mathfrak{r}$ then $a+a^*\in A_+\cap-A_+=\{0\}$ and hence $a=-a^*$, i.e.
\begin{align}
\label{A+=rcapAsa}A_+&=\mathfrak{r}\cap A_\mathrm{sa}.\\
\label{Ask=rcap-r}A_\mathrm{sk}&=\mathfrak{r}\cap-\mathfrak{r}.
\end{align}
If $2\in A^{-1}$ then $a=\frac{1}{2}(a+a^*)+\frac{1}{2}(a-a^*)$ so
\begin{equation}\label{r=A++Ask}
2\in A^{-1}\quad\Rightarrow\quad\mathfrak{r}=A_++A_\mathrm{sk}.
\end{equation}
Also, for $a,b\in A$, $(ba)^*(ba)=a^*b^*ba$ so $a^*|A|^2a\subseteq|A|^2$, $a^*A_\Sigma a\subseteq A_\Sigma$ and $a^*A_+a\subseteq A_+$.  Thus if $b\in\mathfrak{r}$ then $a^*ba+(a^*ba)^*=a^*(b+b^*)a\in a^*A_+a\subseteq A_+$, so we have
\[a^*\mathfrak{r}a\subseteq\mathfrak{r}.\]

As $A_++A_+=A_+$ and $\mathfrak{r}=\mathfrak{r}+\mathfrak{r}$, we get a preorders $\preceq^{\scriptscriptstyle{+}}$ and $\preceq^\mathfrak{r}$ defined by
\begin{align*}
a\preceq^{\scriptscriptstyle{+}}\!b\quad&\Leftrightarrow\quad b-a\in A_+.\\
a\preceq^\mathfrak{r}b\quad&\Leftrightarrow\quad b-a\in\mathfrak{r}.
\end{align*}
\begin{center}
\textbf{From now on $\preceq$ is fixed as an abbreviation for $\preceq^\mathfrak{r}$.}
\end{center}
By \eqref{A+cap-A+}, $\preceq^{\scriptscriptstyle{+}}$ is a partial order which is traditionally only considered on $A_\mathrm{sa}$.  Thus $\preceq$ provides a consistent extension to $A$.
Indeed, \eqref{AsacapAsk} and \eqref{A+=rcapAsa} yield
\begin{align*}
\mathbin{\equiv}\ &=\ \mathbin{=}\quad\ \,\text{on }A_\mathrm{sa}.\\
\mathbin{\preceq}\ &=\ \mathbin{\preceq^{\scriptscriptstyle{+}}}\quad\text{on }A_\mathrm{sa}.
\end{align*}
Also, by \eqref{Ask=rcap-r}, \eqref{r=A++Ask} and $a\equiv\frac{1}{2}(a^*+a)\preceq\frac{1}{2}(b^*+b)\equiv b$, when $2\in A^{-1}$, we have
\begin{align*}
\mathbin{\equiv}\ &=\ \mathbin{\preceq}\cap\mathbin{\succeq}.\\
2\in A^{-1}\quad\Rightarrow\quad\mathbin{\preceq}\ &=\ (\mathbin{\preceq^{\scriptscriptstyle{+}}}\circ\mathbin{\equiv})\ =\ (\mathbin{\equiv}\circ\mathbin{\preceq^{\scriptscriptstyle{+}}})\ =\ (\mathbin{\equiv}\circ\mathbin{\preceq^{\scriptscriptstyle{+}}_\mathrm{sa}}\circ\equiv),
\end{align*}
where $\preceq^{\scriptscriptstyle{+}}_\mathrm{sa}$ denotes the restriction of $\preceq^{\scriptscriptstyle{+}}$ to $A_\mathrm{sa}$.  Also $a^*\mathfrak{r}a\subseteq\mathfrak{r}=\mathfrak{r}^*=\frac{1}{n}\mathfrak{r}$ means
\[nb\preceq nc\quad\Leftrightarrow\quad b^*\preceq c^*\quad\Leftrightarrow\quad b\preceq c\quad\Rightarrow\quad a^*ba\preceq a^*ca.\]
Lastly, denote the composition of $\preceq^\mathfrak{r}$ and $a\mapsto a^*a$ by $\preceq^*$ so
\[a\preceq^*b\quad\Leftrightarrow\quad a^*a\preceq b^*b.\]

\section{Balls and Cones}\label{B&C}

Define the balls $\mathfrak{B}$, $\tfrac{1}{2}\mathfrak{F}$ and $\mathfrak{F}$ and the cone $\mathfrak{c}$ by
\begin{align*}
\mathfrak{B} &=\{a\in A:a^*a\preceq1\}.\\
\tfrac{1}{2}\mathfrak{F} &=\{a\in A:a^*a\preceq a\}.\\
\mathfrak{F} &=\{a\in A:a^*a\preceq2a\}.\\
\mathfrak{c} &=\{a\in A:a^*a\preceq na,\text{ for some }n\in\mathbb{N}\}.
\end{align*}
Note $2a\in\mathfrak{F}\ \Leftrightarrow\ 4a^*a=(2a)^*(2a)\preceq2(2a)=4a\ \Leftrightarrow\ a^*a\preceq a\ \Leftrightarrow\ a\in\frac{1}{2}\mathfrak{F}$, so this is consistent with the fraction notation in \autoref{P}.  Further define operations
\begin{align*}
a^\perp&=1-a.\\
|a|^2&=a^*a.\\
a\bullet b&=a+b-ab.\\
a*b&=a+b-2ab.
\end{align*}
The associativity of $\bullet$ and $*$ follows from the associativity of multiplication and
\begin{align*}
(a\bullet b)^\perp&=1-a-b+ab=a^\perp b^\perp.\\
2(a*b)&=2a+2b-4ab=(2a)\bullet(2b).
\end{align*}
In fact, this shows that $a\mapsto a^\perp$ is a (*-)isomorphism from $(A,\cdot)$ onto $(A,\bullet)$ so $(A,\bullet)$ is also a proper *-semigroup.  We also have the following.
\begin{gather*}
\nonumber\mathfrak{B}^\perp=\mathfrak{F}\subseteq\mathfrak{c}\subseteq\mathfrak{r}.\\
\nonumber\{0,1\}\subseteq|\mathfrak{B}|^2\subseteq\tfrac{1}{2}\mathfrak{F}=(\tfrac{1}{2}\mathfrak{F})^\perp\subseteq\mathfrak{F}\cap\mathfrak{B}.\\
\mathfrak{B}=\mathfrak{B}^*,\quad\tfrac{1}{2}\mathfrak{F}=\tfrac{1}{2}\mathfrak{F}^*,\quad\mathfrak{F}=\mathfrak{F}^*,\quad\mathfrak{c}=\mathfrak{c}^*.\\
\mathfrak{B}\mathfrak{B}=\mathfrak{B},\quad\mathfrak{F}\bullet\mathfrak{F}=\mathfrak{F},\quad\tfrac{1}{2}\mathfrak{F}*\tfrac{1}{2}\mathfrak{F}=\tfrac{1}{2}\mathfrak{F}.\nonumber\\
\mathfrak{c}+\mathfrak{c}=\mathfrak{c},\quad\mathfrak{c}\cap-\mathfrak{c}=\{0\}.\nonumber
\end{gather*}

\begin{proof}\
\begin{itemize}
\item[$\mathfrak{B}^\perp=\mathfrak{F}$]  Note $a^{\perp*}a^\perp=1-a^*-a+a^*a\equiv 1-2a+a^*a$ so $a^*a\preceq2a\Leftrightarrow a^{\perp*}a^\perp\preceq1$.
\item[$\mathfrak{F}\subseteq\mathfrak{c}\subseteq\mathfrak{r}$]  Note $0\preceq a^*a\preceq na$ yields $0\preceq a$.
\item[$\{0,1\}\subseteq|\mathfrak{B}|^2$]  Note $0=0^*0$ and $1=1^{**}=(1^*1)^*=1^*1^{**}=1^*1$.
\item[$|\mathfrak{B}^*|^2\subseteq\tfrac{1}{2}\mathfrak{F}$]  Note $aa^*\preceq1$ yields $(a^*a)^*a^*a=a^*(aa^*)a\preceq a^*a$.
\item[$(\tfrac{1}{2}\mathfrak{F})^\perp=\tfrac{1}{2}\mathfrak{F}$]  Note $a\in\tfrac{1}{2}\mathfrak{F}$ means $a^{\perp*}a^\perp=1-a^*-a+a^*a\preceq a^{\perp*}\equiv a^\perp$.
\item[$\tfrac{1}{2}\mathfrak{F}\subseteq\mathfrak{F}$]  Note $0\preceq a^*a\preceq a$ means $a^*a\preceq2a^*a\preceq2a$.
\item[$\tfrac{1}{2}\mathfrak{F}\subseteq\mathfrak{B}$]  Note $\tfrac{1}{2}\mathfrak{F}=(\tfrac{1}{2}\mathfrak{F})^\perp\subseteq\mathfrak{F}^\perp=\mathfrak{B}$.
\item[$\mathfrak{B}=\mathfrak{B}^*$]  If $a^*a\preceq1$ then $aa^*aa^*\preceq aa^*$ so $0\preceq(aa^*)^{\perp2}=1-2aa^*+aa^*aa^*\preceq(aa^*)^\perp$.
\item[$\frac{1}{2}\mathfrak{F}=\frac{1}{2}\mathfrak{F}^*$ and $\mathfrak{F}=\mathfrak{F}^*$]  Note $\mathfrak{F}=\mathfrak{B}^\perp=\mathfrak{B}^{*\perp}=\mathfrak{F}^*$.
\item[$\mathfrak{c}=\mathfrak{c}^*$]  As above for $\mathfrak{B}=\mathfrak{B}^*$, we have $a^*a\preceq n\ \Leftrightarrow\ aa^*\preceq n$.  Now if $a\in\mathfrak{c}$ then $0\preceq a^*a\preceq na\preceq 2na$.  Then $(n-a)^*(n-a)\equiv n^2-2na+a^*a\preceq n^2$ so $(n-a)(n-a)^*\preceq n^2$ and hence $aa^*\preceq2na^*$.  Thus $a\in\mathfrak{c}^*$.
\item[$\mathfrak{B}\mathfrak{B}=\mathfrak{B}$]  If $a,b\in\mathfrak{B}$ then $a^*a\preceq1$ so $b^*a^*ab\preceq b^*b\preceq1$ hence $ab\in\mathfrak{B}$.
\item[$\mathfrak{F}\bullet\mathfrak{F}=\mathfrak{F}$]  If $a,b\in\mathfrak{F}$ then $a\bullet b=(a^\perp b^\perp)^\perp\in(\mathfrak{B}\mathfrak{B})^\perp=\mathfrak{F}$.
\item[$\tfrac{1}{2}\mathfrak{F}*\tfrac{1}{2}\mathfrak{F}=\tfrac{1}{2}\mathfrak{F}$]  If $a,b\in\frac{1}{2}\mathfrak{F}$ then $2(a*b)=(2a)\bullet(2b)\in\mathfrak{F}\bullet\mathfrak{F}=\mathfrak{F}$ so $a*b\in\frac{1}{2}\mathfrak{F}$.
\item[$\quad\mathfrak{c}\cap-\mathfrak{c}=\{0\}$] If $a^*a\preceq na$ and $a^*a\preceq-ma$ then $(m+n)a^*a\preceq(mn-mn)a=0$.
\item[$\mathfrak{c}+\mathfrak{c}=\mathfrak{c}$] For all $a,b\in A$, we have $0\preceq(a^*-b^*)(a-b)=a^*a-a^*b-b^*a+b^*b$ so
\[a^*b+b^*a\preceq a^*a+b^*b.\]
Thus if $a^*a\preceq na$ and $b^*b\preceq mb$ then
\begin{align*}
(a^*+b^*)(a+b)&\preceq a^*a+a^*b+b^*a+b^*b\\
&\preceq 2(a^*a+b^*b)\\
&\preceq2(ma^*a+nb^*b)\\
&\preceq4mn(a+b).\qedhere
\end{align*}
\end{itemize}
\end{proof}

Thus we get preorders $\preceq^\mathfrak{B}$ and $\preceq^\mathfrak{F}$ and a partial order $\preceq^\mathfrak{c}$ defined by
\begin{align*}
a\preceq^\mathfrak{B}b&\quad\Leftrightarrow\quad a\in\mathfrak{B}b.\\
a\preceq^\mathfrak{F}b&\quad\Leftrightarrow\quad b\in a\bullet\mathfrak{F}.\\
a\preceq^\mathfrak{c}b&\quad\Leftrightarrow\quad b\in a+\mathfrak{c}.
\end{align*}
Noting that $a^*=c^*b^*\ \Leftrightarrow\ a=bc\ \Leftrightarrow\ a^\perp=b^\perp\bullet c^\perp$, we have
\[a^*\preceq^\mathfrak{B}b^*\quad\Leftrightarrow\quad b^\perp\preceq^\mathfrak{F}a^\perp.\]
Also $\mathfrak{c}\subseteq\mathfrak{r}$ and, if $a=cb$ and $c\in\mathfrak{B}$, then $a^*a=b^*c^*cb\preceq b^*b$ so
\begin{align}
\mathbin{\preceq^\mathfrak{c}}\ &\subseteq\ \mathbin{\preceq^\mathfrak{r}}.\label{cr}\\
\mathbin{\preceq^\mathfrak{B}}\ &\subseteq\ \mathbin{\preceq^*}.\label{B|r|^2}
\end{align}
Moreover, \eqref{cr} is almost always a strict inclusion, as $\mathbin{\preceq^\mathfrak{c}}\cap\mathbin{\equiv}$ is $\mathbin{=}$, i.e.
\[\mathfrak{c}\cap A_\mathrm{sk}=\{0\}.\]
For if $a\in A_\mathrm{sk}$ then $a^*a\preceq na\equiv na^*$ implies $2a^*a\preceq n(a+a^*)=0$ and hence $a=0$.  So if $\mathfrak{c}=\mathfrak{r}$ then $A_\mathrm{sk}=\{0\}$, which means $^*$ is the identity and hence $A$ is commutative.  Even this does not guarantee $\mathfrak{c}=\mathfrak{r}$, for example if $A=\mathbb{Z}^\mathbb{N}$ then
\[(1,4,9,\ldots)\in|A|^2\setminus\mathfrak{c}.\]

We should also point out here that in C*-algebras, the various subsets we have defined correspond to their Banach algebra counterparts.  Specifically, with $V(a)$ denoting the numerical range of $a$ (see \cite{BonsallDuncan1973}), for C*-algebra $A$ we have
\begin{align*}
&&&&|A|^2=A_\Sigma=A_+&=\{a\in A:V(a)\subseteq\mathbb{R}_+\}\\
&&&&\subseteq\mathbb{R}_+\mathfrak{F}=\mathfrak{c}\subseteq\mathfrak{r}&=\{a\in A:V(a)\subseteq\mathbb{R}_++i\mathbb{R}\}.\\
&&&&\mathfrak{B}&=\{a\in A:||a||\leq1\}.\\
&&&&\mathbin{\preceq^*}&\subseteq\mathbin{\preceq^\mathfrak{r}}\quad\text{on }A_+.\\
&&&&\mathbin{\preceq^*}&=\mathbin{\preceq^\mathfrak{B}}\quad\text{if $A$ is a von Neumann algebra}.
\end{align*}
For the last two results see \cite{Blackadar2013} Proposition II.3.1.10 and \cite{Pedersen1998} Theorem 2.1.

\section{Orthogonality}\label{O}

We can now say more about the orthogonality relation $\perp$ defined \autoref{S}.
\begin{align}
&&b^*\perp a^*\hspace{4pt}&\Leftrightarrow\quad a\perp b.&&\label{b*a*=0}\\
&&a^*a\perp b\quad&\Leftrightarrow\quad a\perp b.&&\label{aa*b=0}\\
&\text{For }a\in\mathfrak{c}\cup A_+&b^*a\perp b\quad&\Leftrightarrow\quad a\perp b.&&\label{aba*=0}\\[2\jot]
&\text{For }a\in\mathfrak{c}\cup A_+&a\preceq^\mathfrak{r} c\perp b\quad&\Rightarrow\quad a\perp b.&&\label{a<cperpb}\\
&\text{For }a\in\mathfrak{r}&a\preceq^\mathfrak{c}c\perp b\quad&\Rightarrow\quad a\perp b.&&\label{a<{c}cperpb}\\
&\text{For }a\in\mathfrak{r}\text{ and }c\in\mathfrak{c}\cup A_+&a+c\perp b\quad&\Leftrightarrow\quad a\perp b\text{ and }c\perp b.&&\label{a+cperpb}\\[2\jot]
&\text{For }a\in\mathfrak{c}\cup A_\mathrm{n}&a^*\perp b\quad&\Leftrightarrow\quad a\perp b.&&\label{a*perpb}\\
&\text{For }a\in\mathfrak{c}\cup A_+\text{ and }b\in\mathfrak{c}\cup A_\mathrm{n}&ba\perp bc\hspace{6pt}&\Leftrightarrow\quad a\perp bc.&&\label{caba=0}\\
&\text{For }a\in\mathfrak{c}\cup A_\mathrm{n}&a^2\perp b\quad&\Leftrightarrow\quad a\perp b.&&\label{a2b=0}
\end{align}

\begin{proof}\
\begin{itemize}
\item[\eqref{b*a*=0}]  If $ab=0$ then $b^*a^*=(ab)^*=0$.

\item[\eqref{aa*b=0}]  If $a^*ab=0$ then $b^*a^*ab=0$ so $ab=0$, by \eqref{properness}.

\item[\eqref{aba*=0}]  If $b^*ab=0$ and $a^*a\preceq na$ then
\[0\preceq b^*a^*ab\preceq nb^*ab=nb^*0=0.\]
so $b^*a^*ab=0$, by \eqref{A+cap-A+}, thus $ab=0$, by \eqref{properness}.  While if $na=c_1^*c_1+...+c_n^*c_m$ then $b^*c_k^*c_kb=0$, for all $k$, by \ref{antisymmetry} and $nb^*ab=0$.  Then \eqref{properness} yields $c_kb=0$ so $c_k^*c_kb=0$, for all $k$.  Summing yields $nab=0$ and hence $ab=0$.

\item[\eqref{a<cperpb}]  If $a^*a\preceq na$ then $ab=0$ follows from \eqref{properness} and \eqref{A+cap-A+} as
\[0\preceq b^*a^*ab\preceq nb^*ab\preceq nb^*cb=nb^*0=0.\]

If $a\in A_+$ then $0\preceq b^*ab\preceq b^*cb=0$ so $a\perp b$, by \eqref{A+cap-A+} and \eqref{aba*=0}.

\item[\eqref{a<{c}cperpb}]  As $c-a\in\mathfrak{c}$, $(c^*-a^*)(c-a)\preceq n(c-a)$, for some $n$.  Thus
\[b^*(c^*-a^*)(c-a)b\preceq nb^*(c-a)b\preceq nb^*cb=nb^*0=0\]
so $(c-a)b=0$, by \eqref{properness} and \eqref{A+cap-A+}.  Again using $cb=0$, we have $ab=0$.\qedhere

\item[\eqref{a+cperpb}]  If $a\perp b$ and $c\perp b$, certainly $a+c\perp b$.  The converse is \eqref{a<cperpb} and \eqref{a<{c}cperpb}.

\item[\eqref{a*perpb}]  If $a\in\mathfrak{c}$ then $a^*\equiv a\perp b$ yields $a^*\perp b$, by \eqref{a<cperpb}.  While if $a\in A_\mathrm{n}$, this follows from \eqref{properness} and $(ab)^*(ab)=b^*a^*ab=b^*aa^*b=(a^*b)^*(a^*b)$.

\item[\eqref{caba=0}]  By \eqref{a*perpb}, \eqref{properness} and \eqref{aba*=0}, $babc=0\Rightarrow b^*abc=0\Rightarrow c^*b^*abc=0\Rightarrow abc=0$.

\item[\eqref{a2b=0}]  If $a\in\mathfrak{c}$ then $a^*\equiv a\perp ab$ so $a^*ab=0$, by \eqref{a<cperpb}.  If $a\in A_\mathrm{n}$ then $b^*a^*aa^*ab=b^*a^{*2}a^2b=0$ so again $a^*ab=0$, by \eqref{properness}.  Now $ab=0$, by \eqref{aa*b=0}.
\end{itemize}
\end{proof}

Note that \eqref{caba=0} can fail for $a\in A_\mathrm{sa}$, even when $c=1$, for example when $A=M_2$, $a=\begin{bsmallmatrix}0&1\\1&0\end{bsmallmatrix}$ and $b=\begin{bsmallmatrix}1&0\\0&0\end{bsmallmatrix}$ we have $ab=\begin{bsmallmatrix}0&0\\1&0\end{bsmallmatrix}\neq0=bab$.

\begin{cor}\label{orthosym}
Orthogonality is symmetric on $\mathfrak{c}\cup A_\mathrm{n}$.
\end{cor}

\begin{proof}
For $a,b\in\mathfrak{c}\cup A_\mathrm{n}$, we have $a\perp b\ \Leftrightarrow\ b\perp a$ by
\begin{equation}\label{xperpy}
\begin{matrix}
  a\perp b && a\perp b^* & \Leftrightarrow & a^*\!\perp b^* && a^*\!\perp b\\
  \Updownarrow && \Updownarrow\ && \Updownarrow && \ \Updownarrow \\
  b^*\perp a^* & \Leftrightarrow & b\perp a^* && b\perp a & \Leftrightarrow & \ b^*\perp a,
\end{matrix}
\end{equation}
using \eqref{b*a*=0} for the $\Updownarrow$'s and \eqref{a*perpb} for the $\Leftrightarrow$'s.
\end{proof}

\begin{cor}
There are no non-zero nilpotents in $(\mathfrak{c}\cup A_+)(\mathfrak{c}\cup A_+)\cup A_\mathrm{n}$.
\end{cor}

\begin{proof}
Iterating \eqref{caba=0} shows that $(ab)^n=0\ \Rightarrow\ ab=0$, for all $a,b\in\mathfrak{c}\cup A_+$.  By the $b=1$ case, $a\in A_\mathrm{n}$ and $a^n=0\ \Rightarrow\ (a^*a)^n=a^{*n}a^n=0\ \Rightarrow\ a^*a=0\ \Rightarrow\ a=0$.
\end{proof}

In fact, iterating \eqref{caba=0} and \eqref{a2b=0} shows that, for all $a,b\in\mathfrak{c}$ and $l,m,n\in\mathbb{N}$,
\begin{equation}\label{cprod}
a^l(ba)^mb^n=0\quad\Leftrightarrow\quad a\perp b\quad\Leftrightarrow\quad a^lb(ab)^ma^n=0.
\end{equation}
As $A_+^2\subseteq A_\mathrm{sa}^2\subseteq A_+$, this extends to arbitrary products in $A_+$, i.e. whenever $c_1,c_2,\cdots,c_n\in\{a,b\}\subseteq A_+$,
\[c_1c_2\cdots c_n=0\quad\Rightarrow\quad a\perp b.\]
So there are no non-zero nilpotents in the (*-)subsemigroup generated by $a,b\in A_+$.  This also applies to $|A|^2$ for any proper *-semigroup $A$ (see \cite{Bice2015} Corollary 3.6).

Unfortunately, \eqref{cprod} does not extend to arbitrary products in $\mathfrak{c}$.  For every $a\in\mathbb{C}$ has a cube-root $b$ with $\arg(b)\in(-\frac{\pi}{2},\frac{\pi}{2})$, so if $A=\mathbb{C}$ then $A=\mathfrak{c}^3$.  Thus if $A=M_n$ then $A_\mathrm{n}\subseteq\mathfrak{c}^3$, as normal matrices are diagonalizable, by the spectral theorem.  Now by the example mentioned before \autoref{orthosym}, we have $a,b\in\mathfrak{c}$ with $ba^3b=0\neq ab$.

\section{Fixators}\label{F}

For the fixator relation $\ll$ defined in \autoref{S}, we immediately see that
\[a\ll b\quad\Leftrightarrow\quad a=ab\quad\Leftrightarrow\quad b=a\bullet b\quad\Leftrightarrow\quad 0=ab^\perp\quad\Leftrightarrow\quad a\perp b^\perp.\]
Thus the results in \autoref{O} for $\perp$ yield corollaries for $\ll$, e.g. by \eqref{a<cperpb}, \eqref{a<{c}cperpb} and \eqref{a*perpb},
\begin{align}
\label{preceqaux}&\text{For }a\in\mathfrak{c}\cup A_+&a\preceq^\mathfrak{r}c\ll b\quad&\Rightarrow\quad a\ll b.&&\\
\nonumber&\text{For }a\in\mathfrak{r}&a\preceq^\mathfrak{c}c\ll b\quad&\Rightarrow\quad a\ll b.&&\\
\label{a*llb}&\text{For }a\in\mathfrak{c}\cup A_\mathrm{n}&a^*\ll b\quad&\Leftrightarrow\quad a\ll b.&&
\end{align}
Together with \eqref{b*a*=0} and $a^*\preceq b^*\ \Leftrightarrow\ a\preceq b\ \Leftrightarrow\ b^\perp\preceq a^\perp$ (for $\preceq^\mathfrak{c}$ too), we then have
\begin{align}
\label{preceqaux2}&\text{For }b\in\mathfrak{c}^\perp\cup A_+^\perp &a\ll c\preceq^\mathfrak{r}b\quad&\Rightarrow\quad a\ll b.&&\\
\nonumber&\text{For }b\in\mathfrak{r}^\perp&a\ll c\preceq^\mathfrak{c}b\quad&\Rightarrow\quad a\ll b.&&\\
\label{allb*}&\text{For }b\in\mathfrak{c}^\perp\cup A_\mathrm{n}&a\ll b^*\hspace{5pt}&\Leftrightarrow\quad a\ll b.&&
\end{align}
And by \autoref{orthosym}, for all $a\in\mathfrak{c}\cup A_\mathrm{n}$ and $b\in\mathfrak{c}^\perp\cup A_\mathrm{n}$, we have
\[a\ll b\quad\Leftrightarrow\quad a\perp b^\perp\quad\Leftrightarrow\quad b^\perp\!\perp a\quad\Leftrightarrow\quad b^\perp\ll a^\perp.\]
So on $\frac{1}{2}\mathfrak{F}$ and $A_\mathrm{n}$, $a\mapsto a^*$ and $a\mapsto a^\perp$ are $\ll$-isotone and $\ll$-antitone bijections.

We can also replace $\preceq^\mathfrak{r}$ and $\preceq^\mathfrak{c}$ above with $\ll$, $\preceq^\mathfrak{B}$, $\preceq^\mathfrak{F}$ or $\preceq^*$, or even the preorders $\mathbin{\preceq^A}\ \subseteq\ \mathbin{\ll}\cap\mathbin{\preceq^\mathfrak{B}}$ and $\mathbin{\preceq^\bullet}\ \subseteq\ \mathbin{\preceq^\mathfrak{F}}$ defined by
\begin{align*}
a\preceq^Ab\quad&\Leftrightarrow\quad a\in Ab.\\
a\preceq^\bullet b\quad&\Leftrightarrow\quad b\in a\bullet A.
\end{align*}
\begin{align}
\label{preceqA}&&a\preceq^Ac\ll b\quad&\Rightarrow\quad a\ll b.&&\\
\label{BB*r2}&\text{For }b\in\mathfrak{B}& a\preceq^*c\ll b\quad&\Rightarrow\quad a\ll b.&&\\
\label{preceqbullet}&& a\ll c\preceq^\bullet b\quad&\Rightarrow\quad a\ll b.&&\\[2\jot]
&\text{For }b\in\mathfrak{B}& a\ll c\preceq^\mathfrak{B}b\quad&\Rightarrow\quad a\ll b^*b.&&\label{preBaux}\\
&\text{For }b\in\tfrac{1}{2}\mathfrak{F}& a\ll c\preceq^\mathfrak{B}b\quad&\Rightarrow\quad a\ll b.&&\label{Baux}\\
&\text{For }b\in\tfrac{1}{2}\mathfrak{F}\text{ and }c\in\mathfrak{B}& a\ll cb\quad&\Leftrightarrow\quad a\ll b\text{ and }a\ll c.&&\label{allcb}\\[2\jot]
&\text{For }b\in\tfrac{1}{2}\mathfrak{F}& a\ll c\preceq^*b\quad&\Rightarrow\quad a\ll b.&&\label{r2aux}\\
&\text{For }a\in\tfrac{1}{2}\mathfrak{F}\text{ and }b\in\mathfrak{B}& a\preceq^\mathfrak{F}c\ll b\quad&\Rightarrow\quad a\ll b.&&\label{Faux}
\end{align}

\begin{proof}\
\begin{itemize}
\item[\eqref{preceqA}]  See \eqref{preceqll}.

\item[\eqref{BB*r2}]  If $a\preceq^*c\ll b$ then $a^*a\preceq c^*c\preceq^Ac\ll b$ so $a^*a\ll b$, by \eqref{preceqaux} and \eqref{preceqA}, as $a^*a\in|\mathfrak{B}|^2\subseteq\frac{1}{2}\mathfrak{F}\subseteq\mathfrak{c}$.  Then $a^*ab^\perp=0$ gives $ab^\perp=0$, by \eqref{b*a*=0}.

\item[\eqref{preceqbullet}]  Like in \eqref{preceqll}, if $a\ll c\preceq^\bullet b$ then $b=c\bullet d$ so $a\bullet b=a\bullet c\bullet d=c\bullet d=b$.

\item[\eqref{preBaux}]  If $b,c\in\mathfrak{B}$ and $a\ll cb\in\mathfrak{B}\mathfrak{B}=\mathfrak{B}$ then $a\ll b^*c^*$, by \eqref{allb*}.  But $a\ll d,e$ implies $a\ll de$ so $a\ll b^*c^*cb\preceq b^*b$ and hence $a\ll b^*b$, by \eqref{preceqaux2}.

\item[\eqref{Baux}]  By \eqref{preBaux}, $a\ll b^*b\preceq b$ so $a\preceq b$, by \eqref{preceqaux}.

\item[\eqref{allcb}]  If $a\ll c,b$ then $acb=ab=a$.  While if $a\ll cb$ then $a\ll b$, by \eqref{Baux}.  Then $a\ll(cb)^*=b^*c^*$ and $a\ll b^*$ so $a=ab^*c^*=ac^*$, i.e. $a\ll c^*$ so $a\ll c$.

\item[\eqref{r2aux}]  As $c\preceq^*b\in\frac{1}{2}\mathfrak{F}\subseteq\mathfrak{B}$, we have $c^*c\preceq b^*b\preceq1$ so $c\in\mathfrak{B}$ too.  Thus $a\ll1c\preceq^*b$ implies $a\ll c^*c\preceq b^*b\preceq b$, by \eqref{preBaux}, so $a\ll b$, by \eqref{preceqaux}.

\item[\eqref{Faux}]  As $\frac{1}{2}\mathfrak{F}=\frac{1}{2}\mathfrak{F}^*$, $a^*a\preceq a\ \Leftrightarrow\ aa^*\preceq a^*\ \Leftrightarrow\ a\preceq a+a^*-aa^*$ so
\[\tfrac{1}{2}\mathfrak{F}=\{a\in A:a\preceq a\bullet a^*\}.\]
If $a\preceq^\mathfrak{F}c\ll b$ then $a\bullet d\ll b$, for some $d\in\mathfrak{F}$, so $b^\perp\ll a^\perp d^\perp\equiv d^{\perp*}a^{\perp*}$.  Thus $b^\perp\ll a^\perp a^{\perp*}$, by \eqref{preBaux}, so $a\preceq a\bullet a^*\ll b$ and $a\ll b$, by \eqref{preceqaux}.\qedhere
\end{itemize}
\end{proof}

\begin{cor}\label{bigaux}
$\ll$ is auxiliary to $\preceq^\mathfrak{r},\preceq^\mathfrak{c},\preceq^\mathfrak{B},\preceq^\mathfrak{F}$ and $\preceq^*$ on $\frac{1}{2}\mathfrak{F}$.
\end{cor}

\begin{proof}  By the results above, it only remains to show that $\ll\ \subseteq\ \preceq^\mathfrak{r},\preceq^\mathfrak{c},\preceq^\mathfrak{B},\preceq^\mathfrak{F},\preceq^*$ on $\frac{1}{2}\mathfrak{F}$.  Actually $\mathbin{\ll}\subseteq\mathbin{\preceq^\mathfrak{B}}(\subseteq\mathbin{\preceq^*}$ by \eqref{B|r|^2}$)$ is immediate on $\frac{1}{2}\mathfrak{F}\subseteq\mathfrak{B}$, as is $\mathbin{\ll}\subseteq\mathbin{\preceq^\mathfrak{F}}$ on $\frac{1}{2}\mathfrak{F}\subseteq\mathfrak{F}$, remembering that $a\ll b\ \Leftrightarrow\ b=a\bullet b$.  Lastly, for $\ll\ \subseteq\ \preceq^\mathfrak{r},\preceq^\mathfrak{c}$ on $\frac{1}{2}\mathfrak{F}$, if $a,b\in\frac{1}{2}\mathfrak{F}\subseteq\mathfrak{c}$ then $a\ll b$ implies $a^*\ll b$, by \eqref{a*llb}, so
\[(b^*-a^*)(b-a)=b^*b-a-a^*+a^*a\preceq b-a-a^*+a^*=b-a.\]
Thus $b-a\in\frac{1}{2}\mathfrak{F}\subseteq\mathfrak{c}\subseteq\mathfrak{r}$.  Alternatively, by $a\ll b$ and $\frac{1}{2}\mathfrak{F}*\frac{1}{2}\mathfrak{F}=\frac{1}{2}\mathfrak{F}$,
\[b-a=a+b-2a=a+b-2ab=a*b\in\tfrac{1}{2}\mathfrak{F}\subseteq\mathfrak{c}\subseteq\mathfrak{r}.\qedhere\]
\end{proof}

We now examine the $\ll$-lattice structure of subsets containing $A^1_+=\frac{1}{2}\mathfrak{F}\cap A_\mathrm{sa}$.
\begin{align}
\label{cA+lat}A_+\subseteq B\subseteq A\quad&\Rightarrow\quad B\text{ is a $\ll$-semilattice}.\\
\label{halfFA1+lat}A^1_+\subseteq B\subseteq\tfrac{1}{2}\mathfrak{F}\hspace{6pt}&\Rightarrow\quad B\text{ is a $\ll$-lattice}.\\
\label{BA1salat}2\in A^{-1}\quad\text{and}\quad A^1_+\subseteq B\subseteq\mathfrak{B}\quad&\Rightarrow\quad B\text{ is a $\ll$-lattice}.
\end{align}
\begin{proof}\
Iterating \eqref{a+cperpb} and \eqref{allcb}, we see that sums in $\mathfrak{c}\cup A_+$ are $\ll$-supremums and products in $\frac{1}{2}\mathfrak{F}$ are $\ll$-infimums, i.e. (with the product taken in any order)
\begin{align}
\label{llsum}&\text{For finite }F\subseteq\mathfrak{c}\cup A_+&\sum F&=\bigvee F.&&\\
\label{llprod}&\text{For finite }F\subseteq\tfrac{1}{2}\mathfrak{F}&\prod F&=\bigwedge F.&&
\end{align}
As $a^*a+b^*b\in A_+$, for all $a,b\in A$, \eqref{cA+lat} follows from \eqref{aa*b=0} and \eqref{llsum}.  Likewise, as $a^*b^*ba\in|\mathfrak{B}|^2\subseteq A^1_+$, for all $a,b\in\frac{1}{2}\mathfrak{F}\subseteq\mathfrak{B}$, and $a\mapsto a^\perp$ is a $\ll$-antitone bijection, \eqref{halfFA1+lat} follows from \eqref{allb*} and \eqref{llprod}.  If $2\in A^{-1}$ then \eqref{BA1salat} follows from \eqref{halfFA1+lat}, $a=_\ll a^*a\in\frac{1}{2}\mathfrak{F}$ and $a=^\ll\frac{1}{2}(1+a)\in\frac{1}{2}\mathfrak{F}$, for all $a\in\mathfrak{B}$, as
\[a\ll b\quad\Leftrightarrow\quad a\perp b^\perp\quad\Leftrightarrow\quad a\perp\tfrac{1}{2}b^\perp\quad\Leftrightarrow\quad a\ll(\tfrac{1}{2}b^\perp)^\perp=\tfrac{1}{2}(1+b).\qedhere\]
\end{proof}

Thus $\mathbin{\ll}\circ\mathbin{\ll}=\mathbin{\ll}\ \Leftrightarrow\ \mathbin{\ll}_\mathsf{F}\circ\mathbin{\ll}_\mathsf{F}=\mathbin{\ll}_\mathsf{F}$ on $\tfrac{1}{2}\mathfrak{F}$ and $A^1_+$ (and $\mathfrak{B}$ and $A^1_\mathrm{sa}$ if $2\in A^{-1}$).  In fact, it does not matter which subset we consider as $a=^\ll a^*a=_\ll a$, for $a\in\frac{1}{2}\mathfrak{F}$, and, when we identify $B$ with equality on $B$ (i.e. the relation $\mathbin{=}\cap B\times B$),
\[\mathbin{\ll\circ\mathbin{A^1_+}\circ\ll}\ =\ \mathbin{\ll\circ\mathbin{\tfrac{1}{2}\mathfrak{F}}\circ\ll}\ =\ \mathbin{\ll\circ\mathbin{\mathfrak{B}}\circ\ll}.\]
\begin{proof}
If $a\ll b\ll c$ for $b\in\mathfrak{B}$ then $a\ll b^*b\ll c$, by \eqref{aa*b=0} and \eqref{preBaux}.
\end{proof}

For (possibly non-unital) C*-algebra $A$, $\mathbin{\perp}=\mathbin{\ll\circ\perp}$ on $A^1_+$ is the defining property of a SAW*-algebra (see \cite{Pedersen1986}).  As above, we see that $A$ is SAW* iff $\mathbin{\perp}=\mathbin{\ll\circ\perp}$ on $\frac{1}{2}\mathfrak{F}$ or $\mathfrak{B}$ iff $A$ is `Riesz SAW*' in that $\mathbin{\perp_\mathsf{F}}=\mathbin{\ll_\mathsf{F}\circ\perp_\mathsf{F}}$ on $A^1_+$, $\frac{1}{2}\mathfrak{F}$ or $\mathfrak{B}$.  If $A$ is a unital C*-algebra then $\mathbin{\perp}=\mathbin{\ll\circ\perp}$ is equivalent to $\mathbin{\ll}=\mathbin{\ll\circ\ll}$ so
\[A\text{ is SAW*}\quad\Leftrightarrow\quad\ll\text{ has (Riesz) interpolation on }A^1_+,\tfrac{1}{2}\mathfrak{F}\text{ or }\mathfrak{B}.\]

\section{Projections}\label{PJ}

Here we consider the idempotents and projections
\begin{align*}
\mathcal{I}&=\{p\in A:p\ll p\}.\\
\mathcal{P}&=\{p\in A:p\ll p^*\}.
\end{align*}
Note $\mathcal{P}\subseteq|\mathfrak{B}|^2\subseteq A_\mathrm{sa}$ immediately yields $\mathcal{P}=\mathcal{I}\cap|\mathfrak{B}|^2\subseteq\mathcal{I}\cap A_\mathrm{sa}$, even in an arbitrary *-semigroup.  In fact, by \cite{Berberian1972} \S2 Exercise 1A, we have $\mathcal{P}=\mathcal{I}\cap A_\mathrm{n}$, even in an arbitrary proper *-ring (see below).  Thus $\ll$ is a partial order on $\mathcal{P}$, as $\ll$ is reflexive on $\mathcal{I}$ and antisymmetric on $A_\mathrm{sa}$.  Reflexivity combined with auxiliarity on $\mathcal{P}\subseteq\frac{1}{2}\mathfrak{F}$ immediately yields
\begin{align*}
\ll\ &=\ \preceq^\mathfrak{r},\preceq^*,\preceq^\mathfrak{c},\preceq^A,\preceq^\bullet\quad\text{on }\mathcal{P}.\\
\text{Moreover}\quad\mathcal{P}\ &=\ \mathcal{I}\, \cap\, (A_\mathrm{n}\, \cup\, \mathfrak{r}\, \cup\, \mathfrak{r}^\perp\, \cup\, A_\mathrm{sa}A_+^\perp\, \cup\, A_+^\perp A_\mathrm{sa}),
\end{align*}
and hence $\mathcal{P}=\mathcal{I}\cap\mathfrak{B}=\mathcal{I}\cap\mathfrak{F}=\mathcal{I}\cap\frac{1}{2}\mathfrak{F}$, as $|\mathfrak{B}|^2\subseteq\mathfrak{B},\mathfrak{F},\frac{1}{2}\mathfrak{F}\subseteq\mathfrak{r}\cup\mathfrak{r}^\perp$.
\begin{proof}\
\begin{itemize}
\item[($A_\mathrm{n}\cap\mathcal{I}=\mathcal{P}$)]  If $p\in A_\mathrm{n}$ then $p\ll p$ implies $p\ll p^*$, by \eqref{allb*}.

\item[($\mathfrak{r}\cap\mathcal{I}=\mathcal{P}$)]  If $p\in\mathfrak{r}$ then $p+p^*\in A_+$.  If $p\in\mathcal{I}$ too then
\[p^\perp (p+p^*)p^{\perp*}=p^\perp pp^{\perp*}+p^\perp p^*p^{\perp*}=0p^{\perp*}+p^\perp0=0.\]
By \eqref{aba*=0}, $pp^{\perp*}=(p+p^*)p^{\perp*}=0$ so $p=pp^*$.

\item[($\mathfrak{r}^\perp\cap\mathcal{I}=\mathcal{P}$)]  If $a\in\mathcal{I}$ then $a^\perp a^\perp=1-a-a+a=a^\perp$, so $\mathcal{I}=\mathcal{I}^\perp$.  Thus $\mathcal{P}=\mathcal{I}\cap A_\mathrm{n}=\mathcal{I}^\perp\cap A_\mathrm{n}^\perp=\mathcal{P}^\perp$ and hence $\mathcal{P}=\mathcal{P}^\perp=\mathcal{I}^\perp\cap\mathfrak{r}^\perp=\mathcal{I}\cap\mathfrak{r}^\perp$.

\item[($A_\mathrm{sa}A_+^\perp\cap\mathcal{I}=\mathcal{P}$)]  If $a\in A_\mathrm{sa}$, $b\in A_+^\perp$ and $ab=abab$ then $bab^\perp ab=babab-bab=bab-bab=0$ and hence $bab^\perp=0$, by \eqref{aba*=0}, so $ba=bab=(bab)^*=ab\in A_\mathrm{sa}\cap\mathcal{I}=\mathcal{P}$.

\item[($A_+^\perp A_\mathrm{sa}\cap\mathcal{I}=\mathcal{P}$)]  Note $\mathcal{P}=\mathcal{P}^*=(A_\mathrm{sa}A_+^\perp\cap\mathcal{I})^*=A_+^\perp A_\mathrm{sa}\cap\mathcal{I}$.\qedhere
\end{itemize}
\end{proof}

Another fact possibly worth noting is the following.
\begin{equation}\label{plla}
\text{For }p\in\mathcal{P}\text{ and }a\in A_+,\hspace{50pt}p\ll a\quad\Rightarrow\quad p\preceq a.
\end{equation}
\begin{proof}
If $p=pa$ then $p=ap$ so $a-p=a-ap=ap^\perp=ap^{\perp2}=p^\perp ap^\perp\in A_+$.
\end{proof}

We can also use $\mathcal{I}$ to characterize $\ll$, $\perp$ and commutativity on $\mathcal{P}$ as follows.
\begin{align}
\label{pq}pq\in\mathcal{I}\quad&\Leftrightarrow\quad pq=qp.\\
\label{p+q}p+q\in\mathcal{I}\quad&\Leftrightarrow\quad p\perp q.\\
\label{p-q}p-q\in\mathcal{I}\quad&\Leftrightarrow\quad q\ll p.
\end{align}
\begin{proof}  The $\Leftarrow$ parts are immediate, even in an arbitrary *-ring.
\begin{itemize}
\item[\eqref{pq}]  As $A_\mathrm{sa}A_+^\perp\cap\mathcal{I}=\mathcal{P}$, certainly $\mathcal{P}\mathcal{P}\cap\mathcal{I}\subseteq A_\mathrm{sa}$ so $pq=(pq)^*=qp$.
\item[\eqref{p+q}]  If $p+q=(p+q)^2=p+pq+qp+q$ then $pq=-qp$ so $pq=ppq=-pqp\in A_\mathrm{sa}$ and hence $pq=(pq)^*=qp=-pq$ which, as $A$ is torsion-free, means $pq=0$.
\item[\eqref{p-q}]  If $p-q=(p-q)^2=p-pq-qp+q$ then $2q=pq+qp$ so $2pq=pq+pqp$ and hence $pq=pqp=(pqp)^*=qp$.  Thus $2q=2qp$ so again $q=qp$.\qedhere
\end{itemize}
\end{proof}

\section{Products}\label{C}
In this section we make the following additional standing assumption.
\begin{assumption}
\item\label{abA+} $A_+A_+\cap A_\mathrm{sa}=A_+$.
\end{assumption}
As $A_+=\frac{1}{n}A_+$, the apparently weaker assumption $A_\Sigma A_\Sigma\cap A_\mathrm{sa}\subseteq A_+$ would actually suffice.  Also, if $a,b\in A_\mathrm{sa}$ then $ab\in A_\mathrm{sa}\ \Leftrightarrow\ ab=(ab)^*=ba$, so \ref{abA+} is just saying that products of commuting *-positive elements are *-positive (which holds for C*-algebra $A$ \textendash\, see \cite{KadisonRingrose1983} Theorem 4.2.2(iv)).  Using \ref{abA+}, we have the following.
\begin{gather}
\label{A+capB}\qquad\qquad\quad A^1_+=A_+\cap\mathfrak{B}=A_+\cap\mathfrak{r}^\perp.\\
\label{B+C}\qquad\qquad\quad A_+A_+\cap-A_+=\{0\}.\\
\label{-ab=ba}\quad\text{For }a\in A_+\qquad\qquad\quad ab+ba=0\ \quad\Rightarrow\quad a\perp b.\qquad\qquad\qquad\qquad
\end{gather}
\begin{proof}\
\begin{itemize}
\item[\eqref{A+capB}]
Note $A_\mathrm{sa}\cap\tfrac{1}{2}\mathfrak{F}\subseteq A_\mathrm{sa}\cap\mathfrak{r}\cap\mathfrak{B}=A_+\cap\mathfrak{B}=A_+\cap\mathfrak{F}^\perp\subseteq A_+\cap\mathfrak{r}^\perp$.  If $a\in A_+\cap\mathfrak{r}^\perp$ then $a,a^\perp\in A_+$ so $aa^\perp=a^\perp a\in A_+$, by \ref{abA+}, and hence $a^2\preceq a$.

\item[\eqref{B+C}]  Combine \eqref{A+cap-A+} and \ref{abA+}.
\item[\eqref{-ab=ba}]  If $ab=-ba$ then $abb^*=-bab^*\in-A_+$ so $a\perp b$, by \eqref{aa*b=0} and \eqref{B+C}.\qedhere
\end{itemize}
\end{proof}

Actually, from now on, all we need is the strengthening of \eqref{A+cap-A+} given in \eqref{B+C}.

We call $a,b\in A_+$ with $a\perp b$ a \emph{decomposition} of $c\in A_\mathrm{sa}$ if $c=a-b$.  The following generalizes a standard result for C*-algebras (see \cite{KadisonRingrose1983} Proposition 4.2.3(iii)).

\begin{thm}\label{decomp}
Decompositions are unique.
\end{thm}

\begin{proof}
If $a-b=c-d$ and $ab=0=cd$, for some $a,b,c,d\in A_+$, then
\[(a-c)^2b=(b-d)(a-c)b=-(b-d)cb=-bcb\]
so $b\perp c$, by \eqref{aba*=0} and \eqref{B+C}.  Likewise, $a\perp d$ so
\[a^2=a(a-b)=a(c-d)=ac=(a-b)c=(c-d)c=c^2\]
Thus $(a-c)^2=a^2-ac-ca+c^2=0$ so $a=c$, by \eqref{properness}, and hence $b=d$.
\end{proof}

For $B\subseteq A$ let $B'=\{a\in A:\forall b\in S(ab=ba)\}$.  Another standard C*-algebra fact is that any $a\in A_+$ has a *-positive square-root in $C^*(a)\subseteq\{a\}''$, where $C^*(a)$ is the C*-subalgebra generated by $A$.  This generalizes too as follows which, for example, implies that \eqref{plla} extends to $p\in A^{2\perp}_+$.

\begin{thm}\label{psr}
For $a\in A_+$ we have $a\in\{a^2\}''$.
\end{thm}

\begin{proof}
If $ba^2=a^2b$ then $a^2b^*=b^*a^2$ and
\[a(ab-ba)=a^2b-aba=ba^2-aba=(ba-ab)a=-(ab-ba)a.\]
By \eqref{-ab=ba}, $a\perp ab-ba$ so $a^2b=aba$ and $b^*a^2=ab^*a$.  Also $bb^*a^2=ba^2b^*=a^2bb^*$ so the same argument applied to $bb^*$ instead of $b$ yields $a^2bb^*=abb^*a$.  Thus
\[(ab-ba)(b^*a-ab^*)=abb^*a-abab^*-bab^*a+ba^2b^*=a^2bb^*-a^2bb^*-bb^*a^2+bb^*a^2=0.\]
Thus $ab=ba$, by \eqref{properness}.  As $b\in\{a^2\}'$ was arbitrary, $a\in\{a^2\}''$.
\end{proof}

By \autoref{psr}, the positive square-root axiom (PSR) given in \cite{Berberian1972} \S13 Definition 9 reduces to $A_+=A_+^2$ in the presence of \ref{unital*ring} and \eqref{B+C}.  These positive square-roots are even unique, by \cite{Berberian1972} \S13 Exercise 10.  Indeed, if $a,b\in A_+$, $a^2=b^2$ and $ab=ba$ then $(a+b)(a-b)=a^2-ab+ba-b^2=0$ so
\[0=-(a-b)(a+b)(a-b)\preceq(a-b)(a-b)(a-b)\preceq(a-b)(a+b)(a-b)=0.\]
Thus $0=(a-b)^3=(a-b)^4$, by \ref{antisymmetry}, and hence $a=b$, by \eqref{properness}.  Actually, we already have a weak form of (PSR), as \eqref{A+capB} means $a^2\preceq a$, for all $a\in A_+\cap\mathfrak{B}$, so $A_+^2$ is $\preceq$-coinitial in $A_+\cap\mathfrak{B}\setminus\{0\}$.

If $A_+=A_+^2$, define $|a|=\sqrt{a^*a}$.  If $a\in A_\mathrm{sa}$ then $|a|^2=a^*a=a^2$ and hence $|a|a=a|a|$, by \autoref{psr}, so $(|a|+a)(|a|-a)=|a|^2-|a|a+a|a|-a^2=0$, i.e.
\begin{equation}\label{|a|}
|a|+a\perp|a|-a.
\end{equation}
So if $a,-a\preceq|a|$ and $2$ is invertible in $A$ then $\frac{1}{2}(|a|+a)$ and $\frac{1}{2}(|a|-a)$ form a decomposition of $a$.  Also \eqref{|a|} allows us to extend \eqref{B+C} if $\mathbin{\preceq^*}\subseteq\mathbin{\preceq^\mathfrak{r}}$ on $A_+$, yielding an elementary result which might be new even for C*-algebra $A$.

\begin{thm}
If $\mathbin{\preceq^*}\subseteq\mathbin{\preceq^\mathfrak{r}}$ on $A_+=A_+^2$ then $A_+A_+\cap-\mathfrak{r}=\{0\}$.
\end{thm}

\begin{proof}
If $a,b\in A_+$ and $ab\in-\mathfrak{r}$ then $ab\equiv ba\preceq0$ so
\[(a+b)^2=a^2+ab+ba+b^2\preceq a^2-ab-ba+b^2=(a-b)^2.\]
As $\mathbin{\preceq^*}\subseteq\mathbin{\preceq^\mathfrak{r}}$ on $A_+=A_+^2$, we have $a+b\preceq|a-b|$ and hence $2a\preceq|a-b|+a-b$ and $2b\preceq|a-b|+b-a$.  By \eqref{a<cperpb} and \eqref{|a|}, $2a\perp2b$ and hence $a\perp b$.
\end{proof}

\section{Blackadar *-Rings}\label{B*R}

Throughout this section we merely assume
\[A\text{ is a (possibly non-unital) *-ring}.\]

We define \emph{Blackadar} and, for any $R\subseteq A\times A$, \emph{$R$-Blackadar} as follows.
\begin{align}
\label{RB}A\text{ is $R$-Blackadar}\ &\Leftrightarrow\ \mathcal{P}\setminus\{0\}\text{ is $R$-coinitial in }A\setminus\{0\}.\\
\label{B}A\text{ is Blackadar}\ &\Leftrightarrow\ \forall a\in A\setminus\{0\}\hspace{15pt}\exists p\in\mathcal{P}\setminus\{1\}\ (\perp a)\subseteq(\ll p).\\
\text{So}\quad A\text{ is $\subseteq_\perp\!$-Blackadar}\ &\Leftrightarrow\ \forall a\in A\setminus\{0\}\hspace{15pt}\exists p\in\mathcal{P}\setminus\{0\}\ (\perp a)\subseteq(\perp p)\nonumber\\
\text{and}\quad A\text{ is $\ll$-Blackadar}\ &\Leftrightarrow\ (\ll a)\neq\{0\}\Rightarrow\exists p\in\mathcal{P}\setminus\{0\}\ (\ll p)\subseteq(\ll a).\nonumber
\end{align}

Replacing $\subseteq$ with $=$ in \eqref{B} would define a Rickart *-ring (see \cite{Berberian1972} \S3).  Also
\[A\text{ is weakly Rickart}\quad\Leftrightarrow\quad A\text{ is $=_\perp\!$-Blackadar},\]
by \cite{Berberian1972} \S5 Proposition 3 and the following.

\begin{prp}\label{wBproper}
Every $\subseteq_\perp\!$-Blackadar *-ring is proper.
\end{prp}
\begin{proof}
If $a\neq0$ and $aa^*=0$ then we have $p\in\mathcal{P}\setminus\{0\}$ with $p\subseteq_\perp a$ so $pa^*=0=ap$ and hence $p=pp=0$, a contradiction.
\end{proof}

A unital *-ring $A$ is Blackadar iff $A$ is $\subseteq_\perp$-Blackadar, as
\[(\perp a)\subseteq(\ll p\neq1)\quad\Leftrightarrow\quad(\perp a)\subseteq(\perp p^\perp\neq0).\]
In fact, most Blackadar *-rings are automatically unital, as the following generalization of \cite{Berberian1972} \S3 Proposition 2 shows.
\begin{prp}
Every Blackadar *-ring $A$ is proper and
\begin{equation}\label{B*Rprop1eq}
0\in(0\neq)(0\neq)\quad\Rightarrow\quad\mathcal{P}\neq\{0\}\quad\Leftrightarrow\quad0\neq1\in A.
\end{equation}
\end{prp}

\begin{proof}  First we show \eqref{B*Rprop1eq} holds, even under the weaker assumption
\begin{equation}\label{fakeB}
\forall a\in A\setminus\{0\}\ \exists p\in\mathcal{P}\ (\perp a)\subseteq(\ll p).
\end{equation}
For $0\in(0\neq)(0\neq)$ means we have $a,b\neq0=ab$.  Thus we have $p\in\mathcal{P}$ with $(\perp b)\subseteq(\ll p)$ so $a\ll p$.  If $p=0$ then $a=ap=0$, a contradiction, which proves
\[0\in(0\neq)(0\neq)\quad\Rightarrow\quad\mathcal{P}\neq\{0\}.\]

Now if $0\neq p\in\mathcal{P}$ then we have $q\in\mathcal{P}$ with $(\perp p)\subseteq(\ll q)$.  Then, for all $a\in A$, $a=ap+ap^\perp$ (we interpret $ap^\perp$ here as shorthand for $a-ap$) and $ap^\perp\perp p$ so $ap^\perp\ll q$ and hence $a=ap+ap^\perp q=a(p+q-pq)$, i.e. $p+q-pq$ is a right unit for $A$.  In particular, $q=qp+q-qpq$ so $qp=qpq=(qpq)^*=pq$, so $p+q-pq$ is self-adjoint and hence a left unit for $A$ as well, which proves
\[\mathcal{P}\neq\{0\}\quad\Rightarrow\quad0\neq1\in A.\]
The converse is immediate, and in fact this argument shows that
\[\eqref{fakeB}\quad\Leftrightarrow\quad0\notin(0\neq)(0\neq)\text{ or }1\in A.\]
If $0\notin(0\neq)(0\neq)$ then $A$ is certainly proper.  Otherwise $A$ is unital so $A$ is $\subseteq_\perp$-Blackadar and hence proper, by \autoref{wBproper}.
\end{proof}

If $A$ is unital then $a\ll b\ \Leftrightarrow\ a\perp b^\perp$ immediately yields $\mathbin{\subseteq_\ll}\ =\ \mathbin{\subseteq_\perp}$.  In the non-unital case we still have the following.
\begin{align}
\label{subllperp}&&(\subseteq_\ll a)&\subseteq\quad(\subseteq_\perp a)&&\text{if }(a\ll)\neq\emptyset.\\
\label{perpll}&&\subseteq_\perp\quad\ &\subseteq\ \ \quad\subseteq_\ll&&\text{if $A$ is proper}.
\end{align}
\begin{proof}\
\begin{itemize}

\item[\eqref{subllperp}]  If $b\subseteq_\ll a\ll c$ then $b\ll c$.  If $a\perp d$ too then $a\ll cd^\perp$ so $b\ll cd^\perp$.  Thus $b=bcd^\perp=bd^\perp$ so $b\perp d$ and hence $(a\perp)\subseteq(b\perp)$, i.e. $b\subseteq_\perp a$.

\item[\eqref{perpll}]  If $a\subseteq_\perp b\ll c$ then $bc^\perp c^{\perp*}a=0$ so $ac^\perp c^{\perp*}a=0$ and hence $ac^\perp=0$, by properness.  Thus $a\ll c$ and hence $(b\ll)\subseteq(a\ll)$, i.e. $a\subseteq_\ll b$.\qedhere
\end{itemize}
\end{proof}

\[(0\neq\circ\preceq^A)=(0\neq)\text{ and}\preceq^A\!\!\text{-Blackadar}\quad\Rightarrow\quad\subseteq_\perp\!\!\text{-Blackadar}\quad\Rightarrow\quad\ll\!\text{-Blackadar},\]
with equivalence holding if
\begin{equation}\label{0neq}
(0\neq\circ\ll\circ\subseteq_\perp)=(0\neq).
\end{equation}

\begin{proof}
If $(0\neq\circ\preceq^A)=(0\neq)$ then, for any $a\neq0$, we have $b\preceq^Aa$, for some $b\in A\setminus\{0\}$.  If $A$ is $\preceq^A$-Blackadar then $p\preceq^Aa$, for some $p\in\mathcal{P}\setminus\{0\}$ so $p\subseteq_\perp a$, by \eqref{preceqsubsub} and \eqref{preceqperp}, so $A$ is $\subseteq_\perp$\!-Blackadar.

If $A$ is $\subseteq_\perp$-Blackadar and $0\neq b\ll a$ then we have $p\in\mathcal{P}\setminus\{0\}$ with $p\subseteq_\perp b\ll a$ so $p\ll a$, by \autoref{wBproper} and \eqref{perpll}, so $A$ is $\ll$-Blackadar.

If $A$ is $\ll$-Blackadar and \eqref{0neq} holds then, for all $a\neq0$, we have $b,c\in A$ with $c\ll b\preceq^Aa$ so $p\ll b$, for some $p\in\mathcal{P}\setminus\{0\}$.  Thus $p\preceq^Ab\preceq^Aa$ so $p\preceq^Aa$ and hence $A$ is $\preceq^A$-Blackadar.
\end{proof}

In a topological semigroup, define a topological version of the Green relation by
\[a\preceq^\mathsf{L}b\quad\Leftrightarrow\quad a\in\overline{Ab}.\]

\begin{cor}\label{SPBlack}
For C*-algebra $A$,
\[\preceq^A\hspace{-4pt}\text{-Blackadar}\quad\Leftrightarrow\quad\preceq^\mathsf{L}\hspace{-4pt}\text{-Blackadar}\quad\Leftrightarrow\quad\subseteq_\perp\hspace{-4pt}\text{-Blackadar}\quad\Leftrightarrow\quad\ll\hspace{-4pt}\text{-Blackadar}.\]
\end{cor}

\begin{proof}
As multiplication is continuous, $\preceq^A\ \subseteq\ \preceq^\mathsf{L}\ \subseteq\ \subseteq_\perp$.  Thus it suffices to show that $A$ satisfies \eqref{0neq}, which follows from the continuous functional calculus.  Specifically, for any $a\in A\setminus\{0\}$, take continuous functions $f$ and $g$ on $\mathbb{R}$ such that $f\ll g$ and $f(||a||^2)\neq0=g(x)$, for all $x$ in a neighbourhood of $0$, so
\[0\neq f(a^*a)\ll g(a^*a)\preceq^Aa.\qedhere\]
\end{proof}

In C*-algebras, closed left ideals $I$ correspond precisely to hereditary C*-subalgebras $I\cap I^*$.  So $A$ is $\preceq^\mathsf{L}$-Blackadar iff every hereditary C*-subalgebra contains a non-zero projection, which is property (SP) from \cite{Blackadar1994}.  Thus, for C*-algebra $A$,
\begin{equation}\label{wBSP}
A\text{ is $\subseteq_\perp$-Blackadar}\quad\Leftrightarrow\quad A\text{ has property (SP)}.
\end{equation}
Incidentally, for C*-algebra $A$ we also have $\mathbin{\preceq^\mathsf{L}}=\mathbin{\subseteq_{\perp^*}}$, where $\perp^*$ is defined on $A\times A^*$ (here $A^*$ is the dual of $A$) by $a\perp^*\phi\ \Leftrightarrow\ \phi[Aa]=\{0\}$ (see \cite{Effros1963}).

\section{Lattice Structure}\label{LS}

Throughout this section we assume
\[A\text{ is a $\subseteq_\perp$\!-Blackadar *-ring}.\]
Unlike weakly Rickart *-rings, the projections in a $\subseteq_\perp$\!-Blackadar *-ring may not form a lattice.  However, we can still examine supremums and infimums in $\mathcal{P}$ when they do exist, generalizing the weakly Rickart *-ring theory.

First we need the following elementary facts.
\begin{align}
\label{llallb}(\ll a)\cap(\ll b)&\subseteq(\perp ab^\perp).\\
\label{allbll}(a\ll)\cap(b\ll)&\subseteq(ab^\perp\ll).
\end{align}

\begin{proof}\
\begin{itemize}
\item[\eqref{llallb}]  If $c\ll a,b$ then $cab^\perp=cb^\perp=0$.
\item[\eqref{allbll}]  If $a,b\ll c$ then $ab^\perp c=ac-abc=a-ab=ab^\perp$.\qedhere
\end{itemize}
\end{proof}

The following results say $\mathcal{P}$ a complete sublattice of $A$, in an appropriate sense.

\begin{prp}\label{compsublat}
Minimal upper bounds in $\mathcal{P}$ are $\ll$-supremums in $A$.
\end{prp}

\begin{proof}
Say $Q\subseteq\mathcal{P}$ and $Q\ll p\in\mathcal{P}$.  If $p\neq\bigvee Q$ in $A$ then $Q\ll a$ but $p\not\ll a$, for some $a\in A$.  So $0\neq r\subseteq_\perp a^{*\perp}p\ll p$, for some $r\in\mathcal{P}$.  Thus $r\ll p$, by \eqref{perpll}, but
\[Q\subseteq(\ll p)\cap(\ll a)\subseteq(\perp pa^\perp)\subseteq(\perp r).\]
So $Q\ll p-r\ll p$ even though $p\neq p-r\in\mathcal{P}$, i.e. $p$ is not minimal.
\end{proof}

\begin{prp}\label{compsublat2}
Maximal lower bounds in $\mathcal{P}$ are $\ll$-infimums in $A$.
\end{prp}

\begin{proof}
Say $p\ll Q$ but $p\neq\bigwedge Q$ in $A$ so $a\ll Q$ but $a\not\ll p$, for some $a\in A$.  So $0\neq r\subseteq_\perp ap^\perp\perp p$, for some $r\in\mathcal{P}$, and hence $r\perp p$.  But
\[Q\subseteq(p\ll)\cap(a\ll)\subseteq(ap^\perp\ll)\subseteq(r\ll)\]
so $p\ll p+r\ll Q$, even though $p\neq p+r\in\mathcal{P}$, i.e. $p$ is not maximal.
\end{proof}

For $p,q,r\in\mathcal{P}$ we define
\[r=p^\perp\wedge q\quad\Leftrightarrow\quad p\perp r\ll q\quad\text{and}\quad\{s\in\mathcal{P}:p\perp s\ll q\}\subseteq(\ll r).\]
If $A$ is unital, this coincides the definition of $p^\perp\wedge q$ in \eqref{inf}.  In general, we can still characterize $p^\perp\wedge q$ as follows (note $X=Y$ for partially defined expressions $X$ and $Y$ means $X$ is defined iff $Y$ is defined, in which case they coincide).
\begin{equation}\label{pperpwedgeq}
p^\perp\wedge q=\bigvee_{p\perp s\ll q}s.
\end{equation}

\begin{proof}
If $r=p^\perp\wedge q$ then $\bigcap_{p\perp s\ll q}(s\ll)\subseteq(r\ll)$, as $p\perp r\ll q$, and $(r\ll)\subseteq\bigcap_{p\perp s\ll q}(s\ll)$, as $s\ll r$ whenever $p\perp s\ll q$, so $r=\bigvee_{p\perp s\ll q}s$.

Conversely, if $r=\bigvee_{p\perp s\ll q}s$ then $\{s\in\mathcal{P}:p\perp s\ll q\}\subseteq(\ll p^\perp r)$ so $r\ll p^\perp r$, by \autoref{compsublat}.  Thus $rpr=0$ and hence $p\perp r$, by \autoref{wBproper}, so $r=p^\perp\wedge q$.
\end{proof}

Incidentally, for C*-algebra $A$, \eqref{pperpwedgeq} applies even if $A$ is not $\subseteq_\perp$\!-Blackadar.  Indeed, if $r=\bigvee_{p\perp s\ll q}s$ commutes with $p$ then $p^\perp r$ is a projection so the last part still applies even without recourse to \autoref{compsublat}.  While if $r$ does not commute with $p$ then $\sigma(pr)\neq\{0,1\}$ so we can apply the continuous functional calculus as in \cite{Bice2012} to obtain a projection $t\in C^*(r,p)$ with $r\not\ll t$ and $(\ll p^\perp t)=(\ll p^\perp r)$ so $t\in\bigcap_{p\perp s\ll q}(s\ll)\setminus(r\ll)$, contradicting $r=\bigvee_{p\perp s\ll q}s$.

For $R\subseteq A\times A$ and $a\in A$, if we define $p=[a\rrbracket_R\ \Leftrightarrow\ a=_Rp\in\mathcal{P}$ then
\[[a\rrbracket_\ll=[a\rrbracket_\perp.\]
\begin{proof}
If $p=[a\rrbracket_\ll$ then $a=_\ll p\ll p$ and hence $a=_\perp p$, by \eqref{subllperp}, i.e. $p=[a\rrbracket_\perp$.  While if $a=_\perp p$ then $a=_\ll p$, by \autoref{wBproper} and \eqref{perpll}.
\end{proof}
Let $[a\rrbracket=[a\rrbracket_\ll=[a\rrbracket_\perp$, which is the \emph{right support projection} of $a$ (see \cite{Berberian1972} \S3 Definition 4).  Also let $(p\vee q^\perp)\wedge q=(q^\perp\wedge p)^\perp\wedge q$, which is the \emph{Sasaki projection} of $p$ onto $q$ (see \cite{Kalmbach1983} \S7).  By \eqref{Sasaki}, this is coincides with the right support projection of $pq$, while \eqref{pwedgeq} and \eqref{pveeq} generalize \cite{Berberian1972} \S5 Proposition 7.
\begin{align}
\label{Sasaki}(p\vee q^\perp)\wedge q\ &=\ [pq\rrbracket.\\
\label{pwedgeq}p\wedge q\ &=\ [p^\perp q\rrbracket^\perp q.\\
\label{pveeq}p\vee q\ &=\ [pq^\perp\rrbracket+q.
\end{align}

\begin{proof}\
\begin{itemize}
\item[\eqref{Sasaki}]  If $[pq\rrbracket$ is defined then $[pq\rrbracket\ll q$, as $pq\ll q$, so
\[s=[pq\rrbracket^\perp q=q[pq\rrbracket^\perp=q-[pq\rrbracket\in\mathcal{P}.\]
As $pq\ll[pq\rrbracket$, $ps=pq[pq\rrbracket^\perp=0$ so $p\perp s\ll q$.  While if $p\perp r\ll q$, for some $r\in\mathcal{P}$, then $pqr=pr=0$ so $[pq\rrbracket r=0$ and hence $rs=r[pq\rrbracket^\perp=r$, i.e. $r\ll s$.  Thus $s=p^\perp\wedge q$ so $(p\vee q^\perp)\wedge q=q-s=[pq\rrbracket$.

On the other hand, if $p^\perp\wedge q$ is defined then so is
\[s=(p\vee q^\perp)\wedge q=q-(p^\perp\wedge q)=q(p^\perp\wedge q)^\perp=(p^\perp\wedge q)^\perp q.\]
Note $pqs=pq-p(p^\perp\wedge q)=pq$, i.e. $pq\ll s$.  If $s\nsubseteq_\perp pq$ then $pq\perp a$ and $sa\neq0$, for some $a\in A$.  Thus $0\neq r\subseteq_\perp a^*s$, for some $r\in\mathcal{P}$.  Then $r\ll s\ll q$ so $pr=pqr=0$, as $pqsa=pqa=0$, i.e. $p\perp r$.  Thus $r\ll(p^\perp\wedge q)$ so $r=rs=r(q-(p^\perp\wedge q))=r-r=0$, a contradiction.  So $s=[pq\rrbracket$, i.e. $(p\vee q^\perp)\wedge q=[pq\rrbracket$.

\item[\eqref{pwedgeq}]  Note $p^\perp\wedge q=q-((p\vee q^\perp)\wedge q)=$ and $[pq\rrbracket^\perp q=q-[pq\rrbracket$, so all we really need to do is exchange $p$ and $p^\perp$ in the proof of \eqref{Sasaki} above.

\item[\eqref{pveeq}]  If $[pq^\perp\rrbracket$ is defined then $[pq^\perp\rrbracket\perp q$, as $pq^\perp\perp q$, so $q\ll s=[pq^\perp\rrbracket+q\in\mathcal{P}$.  Now $pq[pq^\perp\rrbracket=p0=0$ so $pqs=pqq=pq$ and $pq^\perp q=p0=0$ so $pq^\perp s=pq^\perp[pq^\perp\rrbracket=pq^\perp$.  Thus $ps=pqs+pq^\perp s=pq+pq^\perp=p$, i.e. $p\ll s$.  While if $p,q\ll r$ then $pq^\perp r=prq^\perp=pq^\perp$ so $[pq^\perp\rrbracket\ll r$ and thus $s=[pq^\perp\rrbracket+q\ll r$.  Thus $s=p\vee q$.

If $p\vee q$ is defined let $s=(p\vee q)-q=(p\vee q)q^\perp=q^\perp(p\vee q)\in\mathcal{P}$.  Then $pq^\perp s=pq^\perp(p\vee q)=p(p\vee q)q^\perp=pq^\perp$, i.e. $pq^\perp\ll s$.  If $s\nsubseteq_\perp pq^\perp$ then $pq^\perp a=0\neq sa$, for some $a\in A$.  Thus $0\neq r\subseteq_\perp a^*s$, for some $r\in\mathcal{P}$.  Then $r\ll s\perp q$ so $pr=pq^\perp r=0$, as $pq^\perp sa=pq^\perp a=0$, i.e. $p\perp r$.  Thus $p,q\ll r^\perp(p\vee q)$ and hence $p\vee q\ll r^\perp(p\vee q)$, by \autoref{compsublat}.  Hence $(p\vee q)r(p\vee q)=0$ so $p\vee q\perp r$, by \autoref{wBproper}.  But then $r=rs=r(p\vee q)q^\perp=0$, a contradiction.  Thus $s=[pq^\perp\rrbracket$ so $p\vee q=s+q=[pq^\perp\rrbracket+q$.\qedhere
\end{itemize}
\end{proof}

Let $\top\!_\ll$ and $\top\!_\perp$ denote the $\ll$-incompatibility and $\subseteq_\perp$\!-incompatibility relations.
\begin{align}
\label{SOM}&\text{For }p\in\mathcal{P}&a\not\subseteq^\ll p\ll a\ &\Rightarrow\ \exists q\in\mathcal{P}\backslash\{0\}(p\perp q\ll a).&&\\
\label{SOM*}&\text{For }p\in\mathcal{P}&p\not\subseteq_\ll a\ll p\ &\Rightarrow\ \exists q\in\mathcal{P}\backslash\{0\}(a\perp q\ll p).&&\\
\label{sep}&\text{For }p\in\mathcal{P}&p\not\ll a\ &\Rightarrow\ \exists q\in\mathcal{P}\backslash\{0\}(a\mathbin{\top\!_\ll}q\ll p),&&\\
\label{sep*}&\text{For }p\in\mathcal{P}&p\not\subseteq_\perp a\ &\Rightarrow\ \exists q\in\mathcal{P}\backslash\{0\}(a\mathbin{\top\!_\perp}q\ll p),&&
\end{align}

\begin{proof}\
\begin{itemize}
\item[\eqref{SOM}]  If $a\not\subseteq^\ll p\ll a$ then we have $b\ll a$ with $b\not\ll p$ and hence $bp^\perp\neq0$.  Thus we have a non-zero projection $q\subseteq_\perp bp^\perp$.  As $bp^\perp p=0$, we have $q\perp p$ and, as $bp^\perp a=ba-bpa=b-bp=bp^\perp$, \eqref{perpll} yields $q\ll a$.

\item[\eqref{SOM*}]  If $p\not\subseteq_\ll a\ll p$ then we have $b\gg a$ with $p\not\ll b$ and hence $pb^\perp\neq0$.  Thus we have a non-zero projection $q\subseteq_\perp b^{*\perp}p$ and hence $q\ll p$, by \eqref{perpll}.  As $a\ll p,b$, we have $b^{*\perp}pa^*=b^{*\perp}a^*=0$ and hence $qa^*=0=aq$.

\item[\eqref{sep}]  If $p\not\ll a$, we have a non-zero projection $q\subseteq_\perp a^{*\perp}p$.  By \eqref{perpll}, $q\ll p$ so if $r\ll a,q$ then $a^{*\perp}pr=a^{*\perp}r=0$ and hence $r=qr=0$, i.e. $a\mathbin{\top\!_\ll}q$.

\item[\eqref{sep*}]  If $p\not\subseteq_\perp a$ then $a\perp b$ and $p\not\perp b$, for some $b\in A$.  Thus we have a non-zero projection $q\subseteq_\perp b^*p$.  By \eqref{perpll}, $q\ll p$ so if $r\subseteq_\perp a,q$ then $r\perp b$ and $r\ll p$, by \eqref{perpll}, so $b^*pr=b^*r=0$ and hence $r=qr=0$, i.e. $a\mathbin{\top\!_\perp}q$.\qedhere
\end{itemize}
\end{proof}

There are C*-algebras where \eqref{SOM} and \eqref{SOM*} fail.  For example, considering $C([0,1],M_2)$, every projection $p\neq0,1$ has rank $1$ everywhere on $[0,1]$ and hence the required $q\in\mathcal{P}$ does not exist for $a\neq1$ with $a\not\subseteq^\ll p\ll a$ in \eqref{SOM}, or for $a\neq0$ with $p\not\subseteq_\ll a\ll p$ in \eqref{SOM*}.

If we restrict to $a\in\mathcal{P}$ then \eqref{SOM} and \eqref{SOM*} are just saying that $\mathcal{P}$ is orthomodular, which is immediate (take $q=a-p$ or $p-a$).  On the other hand, there are C*-algebras where $A$ is not orthomodular (w.r.t. $\subseteq_\perp$), e.g. $C([0,1],\mathbb{K})$, where $\mathbb{K}$ denotes the compact operators on a separable infinite dimensional Hilbert space \textendash\, see \cite{AkemannBice2014} Example 4.

Taking $a\in\mathcal{P}$ in \eqref{sep} or \eqref{sep*} generalizes \cite{Bice2009} Theorem 4.4 as follows.

\begin{cor}
Separativity holds on $\mathcal{P}$.
\end{cor}

There are C*-algebras where separativity does not hold on $\mathcal{P}$.  For example, consider $C(X,M_2)$ where $X=\{-1/n:n\in\mathbb{N}\}\cup[0,1]$, and take everywhere rank $1$ projections $p$ and $q$ that coincide on $\{-1/n:n\in\mathbb{N}\}$ but differ on $(0,1]$.  Then $q\mathbin{\top\!_\ll}r\ll p$ implies $r=0$ on $\{-1/n:n\in\mathbb{N}\}$ and hence on $[0,1]$, by continuity.

\newpage

\bibliography{maths}{}
\bibliographystyle{alphaurl}

\end{document}